\documentclass[10pt]{article}
\usepackage[utf8]{inputenc}

\usepackage{amsmath,amsfonts,amsthm,amscd,amssymb,graphicx}
\usepackage{mathabx}
\usepackage[hypertexnames=false]{hyperref}
\usepackage{cleveref}

\usepackage{cite}

\usepackage{cases}

\usepackage{authblk}

\usepackage{bm}

\usepackage{comment}

\usepackage[dvipsnames]{xcolor}

\numberwithin{equation}{section}

\newtheorem{theorem}{Theorem}[section]
\newtheorem{lemma}[theorem]{Lemma}
\newtheorem{proposition}[theorem]{Proposition}

\newtheorem{remark}[theorem]{Remark}

\newcommand{\RR}{\mathbb{R}}

\begin{document}

\title{Asymptotic Behavior of One-Dimensional Hartree Equations \\  with the Long-Range Interaction}

\author[1]{Changhun Yang}
\author[2]{Chanjin You}

\renewcommand{\Affilfont}{\small}

\affil[1]{Department of Mathematics, Chungbuk National University, Cheongju 28644, South Korea\\
	\href{mailto:chyang@chungbuk.ac.kr}{chyang@chungbuk.ac.kr}}

\affil[2]{Department of Mathematics, Yale University, 219 Prospect Street, New Haven, CT 06511, USA\\
	\href{mailto:chanjin.you@yale.edu}{chanjin.you@yale.edu}}

\renewcommand\Authands{, and }

\date{}

\maketitle

\begin{abstract}
We investigate the long-time behavior of small solutions to the one-dimensional Hartree equation with the soft Coulomb interaction. The Hartree equation arises as an effective mean-field model for the evolution of many-body quantum systems, while the soft Coulomb potential provides a regularized one-dimensional approximation of the Coulomb interaction that removes the singularity at the origin while preserving its long-range character.

Owing to the slow spatial decay of the interaction kernel, the equation exhibits long-range nonlinear effects. We prove global existence and modified scattering for sufficiently small, exponentially localized initial data. A notable feature is that, in contrast to the logarithmic phase correction arising in the well-known three-dimensional Coulomb Hartree equation, the one-dimensional soft Coulomb interaction produces a leading phase correction of order \((\log t)^2\), followed by a lower-order logarithmic correction.

The proof is carried out in analytic function spaces. The exponential localization of the initial data is converted into analyticity of the transformed profile, and a method based on  a generator function with a time-dependent radius of analyticity is used to compensate for the derivative loss arising from the nonlinear phase equation. 
\end{abstract}

\section{Introduction}
We consider the following one-dimensional Hartree equation:
\begin{equation}
    \label{eqn:Hartree}
    \begin{cases}
        \begin{aligned}
            &i\partial_t u = - \frac{1}{2}\partial_x^2 u+ \mu \big(w \ast |u|^2\big)u , \\
            &u\vert_{t=0} = u_0.
        \end{aligned}
    \end{cases}
\end{equation}
Here, $u=u(t,x):\mathbb R\times\mathbb R\to\mathbb C$,
$\ast$ denotes the convolution on $\mathbb R$, and the interaction kernel
\[
w(x):=\frac{1}{\sqrt{1+|x|^2}},
\qquad \mu\in\mathbb R\setminus\{0\},
\]
is often referred to as a \textit{soft} Coulomb potential.

\medskip

The soft Coulomb potential regularizes the Coulomb singularity while
retaining the long-range $1/|x|$ behavior. It has long been used in
one-dimensional models of atoms exposed to intense laser fields
\cite{Javanainen1988,Eberly:89}, and subsequently in
one-dimensional electronic-structure and density-functional studies \cite{WagnerPCCP2012,WagnerPRL2012,WagnerPhysRevB2015}. 
It has also been employed in time-dependent density-functional studies of strong-field electron dynamics \cite{PhysRevLett.Thiele} and in quantum-dot nanostructure models \cite{Abdullah2009,Coe_2011}. See also
the references therein. Motivated by these quantum models, we study the corresponding one-dimensional scalar Hartree equation as a natural mean-field model for the many-body dynamics of particles interacting through the soft Coulomb
potential.

\medskip

From an analytical point of view, the regularization of the Coulomb singularity simplifies the Cauchy theory. Indeed, the soft Coulomb potential $w$ belongs to $W^{k,\infty}(\mathbb R)$ for every $k\geq 0$, i.e., $w$ and all of its derivatives are bounded. The averaging effect of the convolution then gives
\[
  \bigl\|\partial_x^k(w*|u|^2)\bigr\|_{L^\infty}
  \lesssim \|\partial_x^kw\|_{L^\infty}\|u\|_{L^2}^2.
\]
Consequently, for every $s\geq 0$,
\[
  \bigl\|(w*|u|^2)u\bigr\|_{H^s}
  \lesssim \|u\|_{L^2}^2\|u\|_{H^s},
\]
and a similar difference estimate shows that the Hartree nonlinearity is locally Lipschitz on $H^s(\mathbb R)$.
A standard fixed-point argument therefore yields local well-posedness in $H^s(\mathbb R)$.
Moreover, combined with the conservation of mass
\[
    \|u(t)\|_{L^2}=\|u_0\|_{L^2},
\]
the above nonlinear estimate yields an a priori bound for the
$H^s$ norm on every finite time interval, thereby giving global
well-posedness in $H^s(\mathbb R)$ for every $s\geq0$.
Thus, the Cauchy theory in Sobolev spaces is relatively straightforward, while the main analytical challenge stems from the long-range $1/|x|$ tail of the soft Coulomb interaction and its effect on the asymptotic dynamics.

\subsection{Scattering and the long-range threshold}

In this paper, we study the long-time behavior of small solutions to \eqref{eqn:Hartree}, with particular emphasis on scattering and precise asymptotic expansions. To explain the long-range threshold, we first consider the Hartree equation
on $\mathbb R^d$ with a general interaction potential:
\begin{equation}
\label{eqn:general-Hartree}
 i\partial_t u
 =
 \left(
 -\frac12\Delta+W\ast |u|^2
 \right)u,
 \qquad
 (t,x)\in\mathbb R\times\mathbb R^d,
\end{equation}
where $W:\mathbb{R}^d\rightarrow \mathbb{R}$. A typical example is the homogeneous potential
\begin{equation}\label{homog.potential}
 W(x)=\kappa |x|^{-\gamma},
 \qquad 0<\gamma<d, \quad \kappa\in\mathbb{R}\setminus\{0\}.
\end{equation}
A particularly important physical example is the Coulomb interaction in three dimensions, corresponding to $\gamma=1$.

By linear scattering in $H^s(\mathbb{R}^d)$, we mean that the nonlinear solution asymptotically behaves like a linear solution: for some $u_{+,\text{LS}}\in H^s(\mathbb{R}^d)$,
\begin{equation}\label{def:LS}
 \| u(t) - e^{it\Delta/2} u_{+,\text{LS}}\|_{H^s(\mathbb{R}^d)} \rightarrow 0 \; \text{ as } \;t\rightarrow \infty,
\end{equation}
where $e^{it\Delta/2}$ denotes the free Schrödinger propagator. Introduce the linear profile $f(t):=e^{-i\frac{t}{2}\Delta} u(t)$. With the Fourier transform convention fixed in the preliminaries, the factorization of the free Schr\"odinger group suggests, for sufficiently regular and localized profiles, the asymptotic formula 
\begin{equation}\label{eqn:free-schrodinger-asymptotics}
 u(t,x)\approx (2\pi it)^{-d/2}e^{i|x|^2/(2t)}
 \widehat f\bigl(t,x/t\bigr),
 \qquad t\to\infty.
\end{equation}
Thus, free solutions spread over the spatial scale $x=t\xi$ with $\xi$ fixed and exhibit the dispersive decay $t^{-d/2}$. By the
unitarity of $e^{it\Delta/2}$ on $H^s$, \eqref{def:LS} is equivalent to
$f(t)\to u_{+,\mathrm{LS}}$ in $H^s$, or, equivalently, to the convergence
of $\langle\xi\rangle^s\widehat f(t,\xi)$ in $L^2_\xi$. By Duhamel's formula, the scattering state is given by
\[
 u_{+,\mathrm{LS}}
 =\lim_{t\rightarrow\infty}f(t)=u_0-i\int_0^\infty e^{-i\tau\Delta/2}
          \bigl(W\ast|u(\tau)|^2\bigr)u(\tau)\,d\tau,
\]
whenever the improper integral converges in $H^s$.
This is the
initial-value scattering problem considered in this paper. Conversely,
the \emph{final-state problem} prescribes $\phi_+\in H^s(\mathbb R^d)$
and asks for a solution satisfying \eqref{def:LS} with
$u_{+,\mathrm{LS}}=\phi_+$. When this solution is uniquely determined,
the map from $\phi_+$ to its initial datum is the forward \emph{wave
operator}; a construction only for sufficiently large times is often
called a local wave operator at infinity. Although our main result concerns
the initial-value problem, we include related wave-operator results in the
literature review below.

\medskip

Differentiating the Fourier profile and extracting the leading term
suggested by \eqref{eqn:free-schrodinger-asymptotics}, we obtain
\begin{equation}\label{eqn:profile-heuristic}
 i\partial_t\widehat f(t,\xi)
 =
 \frac{1}{(2\pi)^d}
 \bigl(W(t\cdot)\ast|\widehat f(t)|^2\bigr)(\xi)
 \widehat f(t,\xi)+R_f(t,\xi),
\end{equation}
where $W(t\cdot)(z):=W(tz)$. The remainder $R_f$ records the difference
between this resonant term and the exact profile equation.\footnote{
The exact expression is $R_f=(2\pi)^{-d}\bigl\{
\mathcal F M(-t)\mathcal F^{-1}
[(W(t\cdot)\ast|\widetilde f|^2)\widetilde f]
-(W(t\cdot)\ast|\widehat f|^2)\widehat f
\bigr\}$, where
$M(t)f(x)=e^{i|x|^2/(2t)}f(x)$ and
$\widetilde f(t,\xi)=(2\pi it)^{d/2}e^{-it|\xi|^2/2}u(t,t\xi)$. }
The rescaling $W(t\cdot)$ comes directly from the spatial scale $x=t\xi$:
writing $y=t\eta$ in the Hartree convolution gives
\[
 (W\ast|u|^2)(t,t\xi)
 \approx
 \frac{1}{(2\pi)^d}
 \int_{\mathbb R^d}W\bigl(t(\xi-\eta)\bigr)
 |\widehat f(t,\eta)|^2\,d\eta.
\]
Here the Jacobian $t^d$ from $y=t\eta$ cancels the factor $t^{-d}$
from $|u(t,t\eta)|^2$. Away from the diagonal $\xi=\eta$, the rescaled
kernel therefore probes the tail of $W$.

\medskip

For the homogeneous potential \eqref{homog.potential}, this becomes
\[
 \frac{1}{(2\pi)^d}
 \bigl(W(t\cdot)\ast|\widehat f(t)|^2\bigr)(\xi)
 =
 \frac{\kappa t^{-\gamma}}{(2\pi)^d}
 \int_{\mathbb R^d}|\xi-\eta|^{-\gamma}
 |\widehat f(t,\eta)|^2\,d\eta.
\]
The kernel $|\xi-\eta|^{-\gamma}$ is locally integrable since
$0<\gamma<d$. If the integral has a nonzero, finite limiting
coefficient $Q_\infty(\xi)$ in a suitable function space, the leading
profile equation takes the form
\[
 i\partial_t\widehat f(t,\xi)
 \approx t^{-\gamma}Q_\infty(\xi)\widehat f(t,\xi).
\]
The relevant distinction is whether the leading time factor
$t^{-\gamma}$ is integrable in time. The regime $\gamma>1$ is
therefore called \emph{short-range}: the leading nonlinear forcing
has a finite time integral, so linear scattering is expected under
suitable assumptions. 
For $1<\gamma<\min\{4,d\}$, both small-data
scattering for the initial-value problem and the existence of wave operators in weighted spaces were proved in \cite{AIHPA1987}. For $\gamma=2$ in dimensions $d>2$, a small-data scattering operator was constructed in \cite{Mochizuki1989OnSD}. For $2<\gamma<d$, small-data
scattering in $H^s(\mathbb R^d)$ for $s\ge\gamma/2-1$ was proved in \cite{miao2008}, assuming smallness in the critical homogeneous
Sobolev norm $\dot H^{\gamma/2-1}$.

On the other hand, $0<\gamma\le1$ is the \emph{long-range} regime:
$t^{-\gamma}$ is not integrable, and one expects the nonlinearity
to leave a persistent phase rather than ordinary linear
scattering. The borderline $\gamma=1$ is the
\emph{scattering-critical} case. For repulsive homogeneous
interactions in dimensions $d\ge2$, nonexistence of nontrivial
linear scattering states in the long-range regime was proved in
\cite{AIHPA1987}; the three-dimensional Coulomb case had already
been studied in \cite{Glassey1997}.

\medskip

By modified scattering, we mean that a real-valued phase
$\Phi(t,\xi)$ can be chosen so that, for some
$u_{+,\mathrm{MS}}\in H^s(\mathbb R^d)$,
\[
 \bigl\|u(t)-e^{-i\Phi(t,-i\nabla)}
                 e^{it\Delta/2}u_{+,\mathrm{MS}}\bigr\|_{H^s(\mathbb R^d)}
 \longrightarrow0
 \qquad\text{as }t\to\infty.
\]
Equivalently, the phase-corrected linear profile $ e^{i\Phi(t,-i\nabla)}f(t)$ converges to
$u_{+,\mathrm{MS}}$ in $H^s(\mathbb R^d)$. The phase generally depends on the solution or its asymptotic profile through nonlinear quantities, and is therefore called
a nonlinear phase correction.

In the homogeneous model \eqref{homog.potential}, the first approximation to the phase is
\[
 \Phi(t,\xi)
 \approx Q_\infty(\xi)\int_1^t \tau^{-\gamma}\,d\tau, \quad 0<\gamma \le 1.
\]
At the scattering-critical exponent $\gamma=1$, the integral is
$\log t$. Modified scattering for small solutions in the critical
Hartree setting was established in
\cite{HNO1998,Hayashi1998a}; modified wave operators, which solve
the corresponding final-state problem with a logarithmic
correction, were constructed in \cite{GO1993}. See also
\cite{Kato2011} for a different approach to critical long-range
scattering.

For $0<\gamma<1$, the leading phase grows like
$(t^{1-\gamma}-1)/(1-\gamma)$. This captures the principal
long-range effect, but for $0<\gamma\le\frac12$ the leading
correction alone leaves a nonintegrable term in the phase equation,
so further higher-order corrections to the asymptotic dynamics are needed.
Small-data asymptotics in dimensions $d\ge2$ under weighted-data
assumptions were studied in \cite{HayashiNaumkin1998,HMJ2001}. On the final-state side, modified
wave operators for $\frac12<\gamma<1$ in dimensions $d\ge3$ were constructed in \cite{GV2000-1}. By incorporating successive higher-order
corrections to both the asymptotic phase and amplitude, they
extended this construction to $0<\gamma\le1$ in dimensions
$d\ge3$ in \cite{GV2000-2}. They subsequently treated low
dimensions using Gevrey spaces in \cite{GV200103}, constructing
local modified wave operators at infinity.

\medskip 

It is worth emphasizing the role of the tail of the potential before
turning to the one-dimensional case. Suppose that $d\ge2$,
$W\in L^1_{\mathrm{loc}}(\mathbb R^d)$, and
\[
    W(x)=A|x|^{-1}+o(|x|^{-1})
    \qquad\text{as }|x|\to\infty
\]
for some $A\neq0$. For
$q\in L^1(\mathbb R^d)\cap L^\infty(\mathbb R^d)$, decompose
\[
    W(t\cdot)\ast q
    =I_{\mathrm{near}}+I_{\mathrm{far}}
\]
according to $|\xi-\eta|\le t^{-1}$ and
$|\xi-\eta|>t^{-1}$. A change of variables gives
\[
    I_{\mathrm{near}}(t,\xi)
    =
    t^{-d}\int_{|z|\le1}
    W(z)q(\xi-z/t)\,dz =
    O\left(
    t^{-d}\|q\|_{L^\infty}
    \|W\|_{L^1(|z|\le1)}
    \right),
\]
whereas the behavior of $W$ at infinity gives
\[
    I_{\mathrm{far}}(t,\xi)
    =
    \frac{A}{t}
    \int_{\mathbb R^d}
    \frac{q(\eta)}{|\xi-\eta|}\,d\eta
    +o(t^{-1})
\]
for each fixed $\xi$. Since $|\cdot|^{-1}$ is locally integrable in
$\mathbb R^d$ for $d\ge2$, the tail of $W$ determines the leading
$t^{-1}$ behavior, while its local behavior contributes only to
lower-order terms. In our application, $q=|\widehat f|^2$, so the
leading coefficient is nonzero for a nontrivial profile.

The situation changes when $d=1$. Although the soft
Coulomb potential itself is bounded at the origin,
\[
    tw(tz)\longrightarrow |z|^{-1}
    \qquad (z\neq0),
\]
and the limiting kernel is not locally integrable at $z=0$. Its
time-dependent regularization produces the additional
logarithmic factor in the asymptotic expansion. We explain this in the next subsection.

\subsection{The scattering-critical soft Coulomb interaction in one dimension}
We now turn to dimension $d=1$. The homogeneous potential
$|x|^{-\gamma}$ is locally integrable near the origin for
$0<\gamma<1$, whereas the formal endpoint $|x|^{-1}$ is not.
In \cite{HMJ1998}, modified scattering for
$\frac12<\gamma<1$ was established using analytic function
spaces to control the derivative loss produced by the
phase--amplitude transformation. For $0<\gamma\le\frac12$,
the same work proved global existence, sharp decay, and an
asymptotic formula. This derivative-loss issue is particularly
delicate in one dimension, where the homogeneous potential
provides weaker high-frequency smoothing. Indeed, for
$0<\gamma<d$,
$\widehat{|x|^{-\gamma}}(\xi)
=c_{d,\gamma}|\xi|^{\gamma-d}$,
so the Hartree convolution gains $d-\gamma$ derivatives at
high frequencies: only $1-\gamma<1$ in one dimension, compared
with more than one derivative in dimensions $d\ge2$ when
$0<\gamma<1$. Related analytic and Gevrey results include
\cite{ASNSP1998,HMJ2001,GV200103}, the last of which concerns
modified wave operators.

\medskip

Related one-dimensional long-range behavior occurs for the cubic nonlinear Schr\"odinger equation, corresponding formally to \eqref{eqn:general-Hartree} with $W=\delta_0$, the Dirac
delta distribution. The cubic equation is also scattering-critical, and a nonexistence result for ordinary linear scattering in the cubic case was established in \cite{Barab1984}. We refer to the results on modified scattering of small solutions in a weighted space \cite{Ozawa1991,Hayashi1998a,Kato2011,Ifrim2015}; see also the survey \cite{Murphy2021}.
By comparison, the power nonlinearity $|u|^{p-1}u$ with $p>3$ is short-range in one dimension. We refer to \cite{tsutsumi1984,Nakamura2002,NO2002,BGTV2023} for linear scattering results in that regime and related settings.

\medskip 

The soft Coulomb potential in \eqref{eqn:Hartree} regularizes the
nonintegrable singularity at the origin while retaining the
$|x|^{-1}$ tail. Unlike the homogeneous kernels
$|x|^{-\gamma}$, whose rescaling produces the exact factor
$t^{-\gamma}$, the rescaled soft Coulomb convolution has an additional logarithmic factor at the endpoint $\gamma=1$.
Set $W_t(z):=\langle tz\rangle^{-1}$. More precisely, as shown in
Lemma~\ref{lem:decomposition-potential}, for
$q\in L^1(\mathbb R)\cap W^{1,\infty}(\mathbb R)$, we have
\begin{equation}\label{eqn:intro-soft-coulomb-expansion}
    (W_t\ast q)(\xi)
    =
    \frac{2\log t}{t}q(\xi)
    +\frac1tF[q](\xi)
    +O_{L^\infty}(t^{-2}),
\end{equation}
where
\begin{equation}\label{eqn:intro-F}
\begin{aligned}
    F[q](\xi)
    &:=
    2\log 2\,q(\xi)
    +\int_{|z|\le1}
      \frac{q(\xi-z)-q(\xi)}{|z|}\,dz 
    +\int_{|z|>1}
      \frac{q(\xi-z)}{|z|}\,dz.
\end{aligned}
\end{equation}
Indeed, integrating the rescaled kernel over $|z|\le1$ gives
\[
    \int_{|z|\le1}\frac{dz}{\sqrt{1+t^2z^2}}
    =
    \frac{2}{t}\operatorname{arsinh}(t)
    =
    \frac{2\log t+2\log 2}{t}+O(t^{-3}).
\]
The regularized core $|z|\le t^{-1}$ contributes only
$O(t^{-1})$, whereas the leading logarithm comes from the
intermediate range $t^{-1}\lesssim|z|\lesssim1$, corresponding to
physical separations $1\lesssim|x-y|\lesssim t$. Thus, the $|x|^{-1}$ tail determines the leading
$t^{-1}\log t$ term, while the lower-order operator $F$ collects
the remaining $t^{-1}$ contributions, including those arising from
the soft-core regularization and from the off-diagonal region
$|z|>1$.

Consequently, formally replacing $q(t,\xi)$ by its asymptotic limit $q_\infty(\xi)$ and integrating the leading term in
\eqref{eqn:intro-soft-coulomb-expansion} in time, we obtain
\[
    \frac{\mu}{2\pi}
    \int_1^t\frac{2\log\tau}{\tau}\,d\tau\,q_\infty(\xi)
    =
    \frac{\mu}{2\pi}(\log t)^2q_\infty(\xi).
\]
This suggests a leading $(\log t)^2$ phase correction.
The $t^{-1}F[q]$ term gives a lower-order $\log t$ correction.
Thus the one-dimensional soft Coulomb interaction differs from
the homogeneous scattering-critical case in dimensions $d\ge2$,
where the leading phase is logarithmic.

\medskip

This feature does not follow by simply taking $\gamma \uparrow 1$ in the one-dimensional homogeneous family. For
$0<\gamma<1$, the leading phase has the polynomial time factor
\[
 \int_1^t\tau^{-\gamma}\,d\tau
 =\frac{t^{1-\gamma}-1}{1-\gamma}
 \longrightarrow \log t
 \qquad\text{as }\gamma\uparrow1
 \quad\text{for each fixed }t.
\]
Taken alone, this limit suggests a logarithmic phase, not a
$(\log t)^2$ phase. However, the coefficient of that time factor is itself singular at the one-dimensional endpoint. For a sufficiently regular density $q$,
\[
 \int_{|z|\le1}|z|^{-\gamma}q(\xi-z)\,dz
 =
 \frac{2}{1-\gamma}q(\xi)+O(1)
 \qquad\text{as }\gamma\uparrow1.
\]
Thus the limit of the time integral alone does not give a limit
for the homogeneous phase, since its spatial coefficient diverges
as $\gamma\uparrow1$. For the soft Coulomb kernel, the singularity
is instead cut off at $|z|\sim t^{-1}$. As shown in
\eqref{eqn:intro-soft-coulomb-expansion}, the leading term of
the rescaled convolution is $\frac{2 \log t }{t}q(\xi)$. Integrating
this term in time gives the leading $(\log t)^2$ phase.

\subsection{Analytic framework and derivative loss}
To make the preceding observation rigorous, we introduce the
following pseudoconformal change of variables, suggested by the
free Schr\"odinger asymptotics
\eqref{eqn:free-schrodinger-asymptotics}:
\[
    v=\frac{x}{t},
    \qquad
    b(t,v):=(2\pi it)^{1/2}e^{-itv^2/2}u(t,tv).
\]
Here, $v=x/t$ is the velocity variable parametrizing the asymptotic rays $x=tv$. This normalization is
used throughout the paper.
Rewriting \eqref{eqn:Hartree} in terms of $b$, we obtain 
\begin{equation}\label{eqn:intro-b}
 i\partial_t b+\frac{1}{2t^2}\partial_v^2b
 =
 \frac{\mu}{2\pi}(W_t\ast_v|b|^2)b,
 \qquad W_t(v)=\langle tv\rangle^{-1}.
\end{equation}
The relation
$\|u(t)\|_{L^\infty_x}
\lesssim t^{-1/2}\|b(t)\|_{L^\infty_v}$
shows that a uniform $H^1_v$ bound on $b$ would yield the sharp
$t^{-1/2}$ decay. Moreover,
$\|\partial_vb(t)\|_{L^2_v}=(2\pi)^{1/2}\|J(t)u(t)\|_{L^2_x}$
for the Galilean vector field $J(t)=x+it\partial_x$.

The standard weighted energy estimate does not by itself provide such a uniform bound. Differentiating \eqref{eqn:intro-b} and using
the fact that $W_t\ast|b|^2$ is real, we obtain
\[
 \frac12\frac{d}{dt}\|\partial_vb(t)\|_{L^2}^2
 =
 \frac{\mu}{2\pi}\operatorname{Im}
 \int_{\mathbb R}
 (W_t\ast\partial_v|b|^2)(v)\,
 b(v)\overline{\partial_vb(v)}\,dv.
\]
The near and far parts of the kernel satisfy
\[
 \|W_t\mathbf1_{|v|\le1}\|_{L^1_v}
 =2t^{-1}\operatorname{arsinh}(t)
 \lesssim t^{-1}\log(2+t),
 \qquad
 \|W_t\mathbf1_{|v|>1}\|_{L^\infty_v}\lesssim t^{-1}.
\]
Together with
\[
 \|\partial_v|b|^2\|_{L^2}
 \lesssim\|b\|_{L^\infty}\|\partial_vb\|_{L^2},
 \qquad
 \|\partial_v|b|^2\|_{L^1}
 \lesssim\|b\|_{L^2}\|\partial_vb\|_{L^2},
\]
Young's inequality yields
\begin{equation}\label{eqn:intro-energy-obstruction}
 \frac{d}{dt}\|\partial_vb(t)\|_{L^2}
 \lesssim
 \frac{\log(2+t)\|b(t)\|_{L^\infty}^2
       +\|b(t)\|_{L^2}^2}{t}
 \|\partial_vb(t)\|_{L^2}.
\end{equation}
Under the desired small-data bootstrap bound
$\|b(t)\|_{L^2\cap L^\infty}\lesssim\epsilon$,
Gronwall's inequality gives only
\[
 \|\partial_vb(t)\|_{L^2}
 \lesssim \|\partial_vb(2)\|_{L^2}
 \exp\!\bigl(C\epsilon^2(\log t)^2\bigr),
 \qquad t\ge2,
\]
which is too weak to close a uniform $H^1_v$ bootstrap. In
fact, this upper bound grows faster than any fixed power of $t$, so it cannot be accommodated by a bootstrap allowing a fixed polynomial loss.

An analogous obstruction to uniform Sobolev bounds arises for the one-dimensional homogeneous potentials $|x|^{-\gamma}$,
$0<\gamma<1$. In the corresponding long-range analyses, one first extracts the leading phase. After this phase correction, however,
the resulting coupled system loses one half derivative in the standard Sobolev energy estimate; see Remark~\ref{rem:derivativeloss}. Analytic and Gevrey spaces were used to compensate for this loss in \cite{HMJ1998,HMJ2001,GV200103}.

Following this strategy, we work with exponentially localized initial data whose pseudoconformal profile is analytic in the
velocity variable. In the analytic energy estimates, a slowly decreasing radius of analyticity compensates for the derivative
loss while remaining uniformly positive. The function
spaces and estimates are given in
Sections~\ref{sec:Analytic Function Space} and~\ref{sec:bootstrap}.

\subsection{Main result}

We now state the main theorem.

\begin{theorem}\label{thm:main}
	Let $\sigma>0$, $\mu \in \mathbb{R} \setminus \{0 \}$, and $\delta >0$. There exists $\epsilon_0 = \epsilon_0(\sigma, \mu ,\delta) >0$ such that if 
    \[
        \varepsilon:=\| e^{\sigma|x|} u_0 \|_{L^2_x} \le \epsilon_0,
    \]
then there exists a unique global-in-time solution $u \in C([0,\infty);L^2(\mathbb{R}))$ satisfying
\begin{equation}\label{est:optimaldecay}
    \| u(t) \|_{L^p_x} \lesssim \varepsilon t^{-1/2+1/p}
\end{equation}
for $2 \le p \le \infty$ and $t \ge 1$. In addition, for some constant $C>0$, it also satisfies
\begin{equation}\label{est:phasemixing}
	\left\| \partial_x^m |u(t)|^2 \right\|_{L^p_x}
	\lesssim \varepsilon^{2} m! \left(\frac{C} \sigma \right)^m t^{-m-1+1/p}
\end{equation}
for every $t\ge1$, $1 \le p \le \infty$, and integer $m \ge 0$.
Moreover, there exists an asymptotic state
$u_\infty\in L^2(\mathbb R)$ such that $\widehat{u_\infty}
    \in \bigcap_{N\ge0}H^N(\mathbb R)$
and 
\begin{equation}\label{est:modifiedscattering}
    \left\| u(t) - e^{it\partial_x^2/2 -i (\log t)^2 A_{\infty}(-i\partial_x) - i (\log t) B_{\infty}(-i\partial_x)} u_{\infty}\right\|_{L^p_x}
    \lesssim \varepsilon\,t^{-3/2+1/p}\log^{5+4\delta}(1+t),
\end{equation}
for $t\ge 1$ and $2 \le p \le \infty$.
Here the real-valued functions $A_\infty$ and $B_\infty$ are defined by
\[
 A_\infty(\xi):=\frac{\mu}{2\pi}|\widehat{u_\infty}(\xi)|^2,
 \qquad
 B_\infty(\xi):=\frac{\mu}{2\pi}
 F[|\widehat{u_\infty}|^2](\xi),
\]
where $F$ is given by \eqref{eqn:intro-F}.
Equivalently, the following asymptotic formula holds:
\begin{equation}\label{eqn:asymptoticformula}
    u(t,x) = (2\pi it)^{-1/2} e^{ix^2/(2t)} e^{-i(\log t)^2 A_{\infty}(x/t)-i(\log t) B_{\infty}(x/t)} \widehat{u_\infty}(x/t) + O_{L^p_x}\big(\varepsilon t^{-3/2+1/p}\log^{5+4\delta}(1+t)\big).
\end{equation}
This holds for $2\le p\le\infty$ as $t\to\infty$.
\end{theorem}

\begin{remark}
\begin{enumerate} 
\item The time-decay rate in \eqref{est:optimaldecay} is shown to be
optimal in the sense that it agrees with that of the free
Schr\"odinger flow.
\item Modified scattering is established in $L^p$ for every $2\le p\le\infty$, with a convergence rate that improves as $p$
increases.
\item Estimate \eqref{est:phasemixing} may be viewed as a phase-mixing bound, in the sense that each spatial derivative of the density gains an additional decay factor of $t^{-1}$.
\item Theorem~\ref{thm:main} holds for every analytic radius
$\sigma>0$ and thus covers the full analytic range. Analyticity
itself, however, may not be optimal for the modified scattering
result. The result might be extended to Gevrey spaces with Fourier
weights
\[
    e^{\sigma | D|^\alpha},
    \qquad 0<\alpha<1.
\]
Such extensions have been established for related long-range Hartree
equations in \cite{ASNSP1998,GV200103}.
\end{enumerate}    
\end{remark}

\subsection{Idea of proof}\label{Idea of Proof}

We adapt the pseudo-conformal amplitude--phase approach of
Hayashi, Kaikina, and Naumkin \cite{HMJ1998}, Hayashi and Naumkin
\cite{HMJ2001}, and Ginibre and Velo \cite{GV200103}.
The distinction from these works is the choice of gauge, or
equivalently, the way the transformed equation is written as a
coupled amplitude--phase system. In those works, the gauge phase
satisfies a Hamilton--Jacobi-type equation driven by the rescaled
Hartree interaction and containing a quadratic phase-gradient
term. Ginibre and Velo explicitly describe this as a
Hamilton--Jacobi gauge condition; in their higher-dimensional
formulation, Hayashi and Naumkin also allow a viscous term in the
phase equation. In \cite{HMJ1998} and in the one-dimensional
argument of \cite{HMJ2001}, the resulting derivative loss is
controlled in analytic spaces, while \cite{GV200103} uses Gevrey
spaces for the construction of modified wave operators. In these
analytic and Gevrey arguments, a decreasing analytic or Gevrey radius provides the favorable term in the weighted energy estimates to absorb the terms with derivative loss.

Here, after the pseudo-conformal transformation
\eqref{eqn:intro-b}, we instead take the full Hartree potential as the time derivative of the real phase:
\[
 \Phi(t,v):=\frac{\mu}{2\pi}
 \int_1^t(W_\tau\ast_v|b(\tau)|^2)(v)\,d\tau,
 \qquad
 a(t,v):=e^{i\Phi(t,v)}b(t,v).
\]
The phase is chosen so that the nonintegrable term no longer
appears in the equation for $a$:
\[
 i\partial_t a
 =
 -\frac{1}{2t^2}\partial_v^2a
 +\frac{i}{t^2}
   \left((\partial_v\Phi)\partial_va
       +\frac12(\partial_v^2\Phi)a\right)
 +\frac{(\partial_v\Phi)^2}{2t^2}a.
\]
We therefore work with a coupled system for $a$ and a
time-normalized phase derivative. This gauge changes the algebraic
form of the system, but it does not eliminate the need to handle
derivative loss.

The coupled equations contain terms with an additional half derivative.
To control these terms, we allow the analyticity radius
$z=\lambda(t)$ to decrease in time, as in
\cite{HMJ2001,GV200103}. We formulate the argument in terms of the generator function $G(t,z)$,
following Grenier, Nguyen, and Rodnianski
\cite{GrenierNguyenRodnianski2021} and Gagnebin and Iacobelli
\cite{GagnebinIacobelli2023}. Indeed,
$\partial_zG$ controls the terms containing an additional half derivative,
while differentiating $G(t,\lambda(t))$ produces the term
$\lambda'(t)\partial_zG(t,\lambda(t))$. We choose $-\lambda'(t)$ to dominate
the coefficient of $\partial_zG$, thereby absorbing the terms with
derivative loss. In our case, the required rate of decrease of $\lambda(t)$
is integrable in time. For sufficiently small data, the analyticity radius remains uniformly
positive, and the bootstrap estimates close globally.

These bounds imply convergence of $a(t)$ in a slightly smaller
analytic space. Finally, the decomposition
\eqref{eqn:intro-soft-coulomb-expansion} identifies both the
$(\log t)^2$ and $\log t$ terms of $\Phi$. The remaining bounded
phase is absorbed into $\widetilde a_\infty$, yielding the
asymptotic formula in Theorem~\ref{thm:main}.

\subsection{Future directions: one-dimensional density-matrix equations}

The scalar Hartree equation \eqref{eqn:Hartree} considered in this paper is
the simplest scalar mean-field model for many-body quantum dynamics with
the soft Coulomb interaction. More physically relevant models are provided
by density-matrix formulations, most notably the time-dependent
Hartree--Fock and Kohn--Sham equations, which play central roles in
electronic-structure theory and density-functional theory. One-dimensional
many-electron and density-functional models based on the soft Coulomb
interaction have been extensively investigated in
\cite{WagnerPCCP2012,WagnerPRL2012,PhysRevLett.Thiele}; see also the
references therein.

For a nonnegative trace-class density matrix $\gamma$ on $L^2(\mathbb R)$,
write $\rho_\gamma(x)=\gamma(x,x)$, and let $X_\gamma$ be the exchange
operator with kernel
\[
 X_\gamma(x,y)=w(x-y)\gamma(x,y).
\]
The Hartree--Fock equation is given by
\begin{equation}\label{eqn:HartreeFock}
 i\partial_t\gamma
 =
 \bigl[-\tfrac12\Delta+w\ast\rho_\gamma-X_\gamma,\gamma\bigr].
\end{equation}
It is formally equivalent to the infinite coupled system
\begin{equation}\label{coupledNLS}
 i\partial_tu_j
 =
 -\tfrac12\Delta u_j
 +\sum_{k=1}^{\infty}\nu_k(w\ast|u_k|^2)u_j
 -\sum_{k=1}^{\infty}
 \nu_k\bigl(w\ast(u_j\overline{u_k})\bigr)u_k,
 \qquad j=1,2,\ldots,
\end{equation}
where
\[
 \gamma(t)
 =
 \sum_{j=1}^{\infty}
 \nu_j\lvert u_j(t)\rangle\langle u_j(t)\rvert.
\]
Here, $\{u_j(t)\}_{j=1}^{\infty}$ is an orthonormal system in
$L^2(\mathbb R)$ for each $t$, $\nu_j\geq0$, and
$\sum_{j=1}^{\infty}\nu_j<\infty$. The scalar Hartree equation
\eqref{eqn:Hartree}, with $\mu = \nu_1 >0$, is obtained from the direct Hartree part of
\eqref{coupledNLS} by retaining a single orbital and omitting the exchange
term.

In velocity variables, applying
\eqref{eqn:intro-soft-coulomb-expansion} componentwise gives a 
leading local contribution of order $(2\log t)/t$ to the direct term, proportional
to $(\sum_k\nu_k|b_k|^2)b_j$. The exchange term gives the same local contribution with the opposite sign. These terms therefore
cancel in the formal asymptotic expansion;
equivalently, the combined Hartree--Fock integrand vanishes on the diagonal,
which removes the formal source of the $(\log t)^2$ phase found
in the scalar Hartree equation.

Exploiting an analogous direct--exchange cancellation, Mal\'ez\'e
\cite{Maleze2025} proved linear scattering in one dimension for small,
localized data when the interaction potential is a finite measure.
This raises the natural question of whether an analogous cancellation 
can be rigorously justified for the long-range, nonintegrable soft Coulomb interaction and, if so, whether the remaining contribution produces a logarithmic phase or undergoes further cancellation leading to linear scattering.

For Kohn--Sham dynamics, the direct Hartree Hamiltonian is instead supplemented by a local exchange--correlation potential $V_{\mathrm{xc}}[\rho_\gamma]$:
\begin{equation}\label{eqn:KohnSham}
 i\partial_t\gamma
 =
 \bigl[-\tfrac12\Delta+w\ast\rho_\gamma
       +V_{\mathrm{xc}}[\rho_\gamma],\gamma\bigr].
\end{equation}
The effect of this local term on the long-time behavior also remains to be studied for the soft Coulomb interaction. We refer to \cite{Pusateri2021} for short-range scattering
in time-dependent Kohn--Sham models and to \cite{KLYY2026arxiv} for recent
modified scattering results in critical two- and three-dimensional models.

\subsection*{Organization of the paper}
The remainder of the paper is organized as follows. In
Section~\ref{sec:preliminaries}, we introduce the pseudoconformal variables,
derive the coupled system for the modified profile and normalized phase
derivative, and collect the analytic-space estimates used below.
Section~\ref{sec:local-wellposedness} establishes local well-posedness and
propagates the exponentially weighted norm to $t=1$, thereby providing
the initial analytic control for the transformed system. In
Section~\ref{sec:nonlinear-estimates}, we derive analytic energy estimates,
close the global bootstrap using a slowly decreasing radius of
analyticity, and prove the dispersive and phase-mixing bounds. Finally, in
Section~\ref{sec:modified-scattering}, we prove convergence of the modified
profile, obtain a two-term asymptotic expansion for the rescaled soft Coulomb
interaction, and identify the $(\log t)^2$ and $\log t$ phase corrections,
completing the proof of Theorem~\ref{thm:main}.

\subsection*{Acknowlegements}
The work of C.~Yang was supported by a funding for the academic research program of Chungbuk National University in 2026 and by the National Research
Foundation of Korea (NRF) grant funded by the Korea government (MSIT) (RS-2026-25587425).
The authors acknowledge the use of AI tools for proofreading and language editing. The authors reviewed all revisions independently, and take full responsibility for the mathematical content of the manuscript.

\section{Preliminaries}\label{sec:preliminaries}
In this section, we introduce the pseudoconformal variables and analytic norms used throughout the paper, and record some lemmas needed in the nonlinear estimates.

\subsection{Notations}

The Fourier transform and its inverse are defined by
\[
\mathcal{F}f(k) = \widehat{f}(k) = \int_{\RR} e^{-i k x} f(x) \, d x, \qquad \mathcal{F}^{-1} g(x) = \frac{1}{2\pi} \int_{\RR} e^{ikx} g(k) \, dk.
\]
We define for $t>0$ that
\[
U(t) =e^{i t \partial_x^2/2} , \qquad 
(M(t) f)(x) = e^{i x^2/(2t)} f(x), \qquad (D(t) f)(x) = (2\pi i t)^{-1/2} f(x/t).
\]
Then the free Schr\"odinger group admits the factorization
\[
U(t) = M(t) D(t)\mathcal{F}M(t).
\]
We will also use the Galilean vector field
\[
    J(t) = x + it\partial_x = U(t) x U(-t) = M(t) (it\partial_x) M(t)^{-1},
\]
We then define $e^{c J(t)}$ for $c \in \mathbb{R}$ using functional calculus. Namely,
\[
    e^{c J(t)} = U(t) e^{cx} U(-t).
\]

\subsection{Pseudoconformal transformation}
For $t >0$, we consider the pseudoconformal transformation
\[
    u(t) = M(t) D(t) b(t).
\]
Equivalently, we set $v=x/t$ and define the profile $b=b(t,v)$ by
\[
	b(t,v) = (2\pi it)^{1/2}e^{-itv^2/2} u(t,tv).
\]
Recalling that $u$ solves \eqref{eqn:Hartree}, a direct computation gives
\begin{equation}\label{eqn:Hartree-pseudoconformal}
	i\partial_t b + \frac{1}{2t^2} \partial_v^2 b = \frac{\mu}{2\pi} (W_t \ast_v |b|^2)b,
\end{equation}
where
\begin{equation}\label{def:potential-pseudoconformal}
	W_t(v) := w(tv) = \langle tv \rangle^{-1}.
\end{equation}

Since $W_t \ast_v |b(t)|^2$ is expected to decay at a rate $t^{-1} \log t$, we absorb this term to a phase. Namely, we define 
\[
	\Phi(t,v) = \frac{\mu}{2\pi} \int_1^t  (W_\tau \ast_v |b(\tau)|^2)(v) \, d\tau
\]
which is real-valued. Define the modified profile
\[
	a(t,v):= e^{i\Phi(t,v)} b(t,v).
\]
Denote $\phi(t,v) =\partial_v \Phi(t,v)$ for notational convenience. The equation for $a$ becomes
\begin{equation}\label{eqn:a-1}
	i\partial_t a 
	= -\frac{1}{2t^2} \partial_v^2 a  + \frac{i}{t^2} \left(\phi \partial_v a 
	+ \frac{1}{2} (\partial_v \phi)a\right) +\frac{1}{2t^2} \phi^2 a.
\end{equation}
Moreover, $\phi$ solves
\begin{equation}\label{eqn:phi-1}
	\partial_t \phi = \frac{\mu}{2\pi} \partial_v (W_t \ast |a|^2), \qquad \phi(t=1)=0.
\end{equation}
One expects
	\[
		\Phi(t,v) = O((\log t)^2), \qquad \phi(t,v) = O((\log t)^2).
	\]
We normalize $\phi$ by introducing
\[
	h(t,v) = \Theta(t)^{-1-\delta} \phi(t,v), \qquad \Theta(t) = \log^2 (1+t)
\]
for $t \ge 1$ and for some fixed $\delta >0$. Note that $h(1)=0$. We allow a logarithmic loss by fixing $\delta >0$ to ensure that the coefficient in the energy estimate is integrable in time, see Section \ref{sec:bootstrap}.
Then \eqref{eqn:a-1} becomes
\begin{equation}\label{eqn:a-2}
	i\partial_t a 
	= -\frac{1}{2t^2} \partial_v^2 a  + \frac{i\Theta(t)^{1+\delta}}{t^2} \left(h \partial_v a 
	+ \frac{1}{2} (\partial_v h)a\right) +\frac{\Theta(t)^{2+2\delta}}{2t^2} h^2 a.
\end{equation}
Also, \eqref{eqn:phi-1} becomes
\begin{equation}\label{eqn:phi-2}
	\partial_t h + (1+\delta)\frac{\Theta'(t)}{\Theta(t)} h =  \frac{\mu}{2\pi} \Theta(t)^{-1-\delta} \partial_v (W_t \ast_v |a|^2).
\end{equation}

\begin{remark}\label{rem:derivativeloss}
Then if we perform energy estimate for the normalized system for $(a,h)$ by setting
\[
    E_s(t) := \| a(t) \|_{H^s}^2 + \| h(t) \|_{H^s}^2,
\]
then one obtains the estimate of the form
\[
    \frac{d}{dt} E_s(t)
    \lesssim
    \frac{\log^{4+4\delta}(1+t)}{t^2} E_s(t)^2
    +
    \frac{E_s(t)^{1/2}
    \big(E_s(t)+D_s(t)\big)}{t\log^{1+2\delta}(1+t)} ,
\]
where
\[
    D_s (t)
    = 
    \| |\partial_v|^{1/2} a(t) \|_{H^s}^2
    +
    \| |\partial_v|^{1/2} h(t) \|_{H^s}^2.
\]
This is indeed a derivative loss in the sense that it cannot be closed for a fixed $s$. This does not rule out a different argument in Sobolev spaces.
\end{remark}

\subsection{Analytic function space}\label{sec:Analytic Function Space}
For $z \ge 0$ and $s \ge 0$, we introduce the analytic norm
\[
	\| f \|_{z,s}^2
    := \| e^{z|\partial_v|} \langle \partial_v \rangle^s f \|_{L^2_v}^2
    =
    \frac{1}{2\pi} \int_{\RR} e^{2z|\eta|} \langle \eta \rangle^{2s}|\widehat{f}(\eta)|^2 \, d\eta.
\]
We denote the corresponding space as $\mathcal{G}^{z,s}(\mathbb{R})$.
The inner product $\langle \cdot, \cdot\rangle_{z,s}$ has the same Fourier weight.
The Sobolev correction $\langle \partial_v\rangle^s$ allows us to estimate products without loss of analyticity radius. We assume $s>1$ to ensure that $\mathcal{G}^{z,s-1/2}(\mathbb{R})$ is an algebra.

We record some properties of the analytic function space.
\begin{lemma}\label{lem:analyticspace}
	Let $r\ge s >1/2$ and $z \ge 0$.
	Then 
	\[
	\| fg \|_{z,r} \lesssim \|f \|_{z,r}  \|g \|_{z,s} + \| f \|_{z,s} \| g \|_{z,r}.
	\]
	In particular, $\mathcal{G}^{z,s}(\mathbb{R})$ is an algebra, that is,
	\[
	\| f g \|_{z,s} \lesssim_s \| f \|_{z,s} \| g \|_{z,s}.
	\]
	Moreover,
	\[
		\left| \langle f ,g \rangle_{z,s} \right| \lesssim \| f \|_{z,s-1/2} \| g \|_{z,s+1/2}.
	\]
	Finally,
	\[
		\| f \|_{z,s+1/2}^2 \lesssim \| f \|_{z,s}^2 + \| |\partial_v|^{1/2} f \|_{z,s}^2.
	\]
\end{lemma}
\begin{proof}
	Note that
	$e^{z|\xi|}  \le e^{z|\eta|} e^{z|\xi-\eta|}$
	and
	$
		\langle \xi \rangle^r \lesssim \langle \eta \rangle^r + \langle \xi-\eta \rangle^r
	$.
	Therefore, we get
	\[
		e^{z|\xi|} \langle \xi \rangle^r \lesssim e^{z|\eta|} \langle \eta \rangle^r e^{z|\xi-\eta|} + e^{z|\eta|} e^{z|\xi-\eta|} \langle \xi -\eta \rangle^r.
	\]
	It follows that
	\begin{align*}
		e^{z|\xi|} \langle \xi \rangle^r |\widehat{fg}(\xi)|
		&\lesssim  \int e^{z|\eta|}\langle \eta \rangle^r |\widehat{f}(\eta)| e^{z|\xi-\eta|} |\widehat{g}(\xi-\eta)| \, d\eta \\
		&\qquad + \int e^{z|\eta|} |\widehat{f}(\eta)|e^{z|\xi-\eta|} \langle \xi-\eta \rangle^r  |\widehat{g}(\xi-\eta)| \, d\eta.
	\end{align*}
	Taking $L^2_{\xi}$ on both sides and applying Young's inequality, we obtain that 
	\[
		\| fg \|_{z,r} \lesssim \| f \|_{z,r} \| e^{z|\xi|}\widehat{g} \|_{L^1_{\xi}} + \| e^{z|\xi|} \widehat{f} \|_{L^1_{\xi}}\| g \|_{z,r}.
	\]
	Since $s>1/2$, we get $\langle \xi \rangle^{-s} \in L^2(\mathbb{R})$. Then we have
	\[
		\| e^{z|\xi|} \widehat{f} \|_{L^1} \lesssim \| \langle \xi \rangle^{-s} \|_{L^2} \| f \|_{z,s} \lesssim \| f \|_{z,s}.
	\]
	We thus obtain that
	\[
		\| fg \|_{z,r} \lesssim \|f \|_{z,r}  \|g \|_{z,s} + \| f \|_{z,s} \| g \|_{z,r}
	\]
	for $r \ge s >1/2$. In particular, substituting $r=s$, we obtain that $\mathcal{G}^{z,s}(\mathbb{R})$ is an algebra.
	
	Next, using $2s= (s-1/2) + (s+1/2)$ and applying Cauchy--Schwarz, we get
	\begin{align*}
		|\langle f,g\rangle_{z,s}|
		&=
		\frac{1}{2\pi}
        \left|
		\int e^{2z|\eta|}
		\langle \eta \rangle^{2s} \widehat f(\eta)\overline{\widehat g(\eta)} \, d\eta
		\right| \\
		&\le \frac{1}{2\pi}
		\left(
		\int e^{2z|\eta|} \langle \eta \rangle^{2s-1} |\widehat{f}(\eta)|^2 \, d\eta
		\right)^{1/2} 
		\left(
		\int e^{2z|\eta|} \langle \eta \rangle^{2s+1} |\widehat{g}(\eta)|^2 \, d\eta
		\right)^{1/2} \\
		&=
		\|f\|_{z,s-1/2}\,
		\|g\|_{z,s+1/2}.
	\end{align*}
	Finally, the last bound in the lemma follows from $\langle \eta \rangle^{2s+1} \le \langle \eta \rangle^{2s}+|\eta| \langle \eta \rangle^{2s}$.
    This completes the proof of Lemma \ref{lem:analyticspace}.
\end{proof}

\section{Local well-posedness}\label{sec:local-wellposedness}
In this section, we establish local well-posedness and propagate the weighted norm
\[
    \| e^{\lambda_0 |J(t)|} u(t) \|_{L^2_x}
\]
on a time interval containing $t=1$. In particular, this yields the initial bound at $t=1$ needed for the analysis using pseudoconformal transformation.

\begin{lemma}\label{lem:LWP}
    Let $\lambda_0 > \lambda_1>0$ and $T_0 >1$. Define
    \[
        R= \| e^{\lambda_0 |x|} u_0 \|_{L^2_x} < \infty.
    \]
    There exists $\epsilon_0 >0$ such that if $R \le \epsilon_0$, then \eqref{eqn:Hartree} admits a unique solution $u \in C([0,T_0];L^2)$ satisfying
    \[
        \sup_{t \in[0,T_0]} \| e^{\lambda_0|J(t)|}  u(t) \|_{L^2_x} \le 2R.
    \]
    Consequently, for $s \ge 0$, we have
    \[
        \sup_{t \in[0,T_0]} \| e^{\lambda_1 |J(t)|} \langle J(t) \rangle^s  u(t) \|_{L^2} \le 2C R
    \]
    for some $C>0$ depending only on $s$ and $\lambda_0 - \lambda_1$.
\end{lemma}

\begin{proof}
    Let $\mathcal{N}(u_1,u_2,u_3)= (w \ast (u_1 \overline{u_2})) u_3$, $N(u) =\mathcal{N}(u,u,u)$, and $\widetilde{u}_j(t) = M(t)^{-1} u_j(t)$ for $j=1,2,3$.  Then we have 
    $$
        M(t)^{-1}\mathcal{N}(u_1(t),u_2(t),u_3(t))
        =
        \big(w \ast (\widetilde{u}_1(t) \overline{\widetilde{u}_2(t)})\big)\widetilde{u}_3(t)
        =
        \mathcal{N}(\widetilde{u}_1(t), \widetilde{u}_2(t), \widetilde{u}_3(t)).
    $$
    We often omit the dependence on $t$ to simplify the notation. We obtain that
    \begin{align*}
        J\mathcal{N}(u_1, u_2, u_3)
        &= 
        M(t) (it\partial_x) 
        \mathcal{N}(\widetilde{u}_1, \widetilde{u}_2, \widetilde{u}_3) \\
        &=
        M(t) 
        \Big(
        (w \ast (it\partial_x(\widetilde{u}_1 \overline{\widetilde{u}_2})))\widetilde{u}_3 + (w \ast (\widetilde{u}_1 \overline{\widetilde{u}_2})) it\partial_x \widetilde{u}_3
        \Big).
    \end{align*}
    Using 
   $ 
        it\partial_x (\widetilde{u}_1\overline{\widetilde{u}_2})
        = 
        (it\partial_x \widetilde{u}_1) \overline{\widetilde{u}_2}
        - \widetilde{u}_1 \overline{it\partial_x \widetilde{u}_2},
    $
    and $J(t)u_j(t)=M(t)it\partial_x \widetilde{u}_j(t)$, 
    we get
    \[
        J\mathcal{N}(u_1, u_2, u_3)
        =\mathcal{N}(Ju_1, u_2, u_3) -\mathcal{N}(u_1, Ju_2, u_3) + \mathcal{N}(u_1, u_2, Ju_3).
    \]
    Applying $J(t)$ repeatedly, we get
    \[
        J^n \mathcal{N}(u_1,u_2,u_3) = \sum_{\alpha_1 + \alpha_2 + \alpha_3= n} (-1)^{\alpha_2} \frac{n!}{\alpha_1! \alpha_2! \alpha_3!} \mathcal{N}(J^{\alpha_1} u_1, J^{\alpha_2} u_2, J^{\alpha_3} u_3).
    \]
    Using the identity
    \[
        e^{c J(t)} = \sum_{n=0}^{\infty} \frac{c^n}{n!} J(t)^n,
    \]
    for $c \in \mathbb{R}$, it follows that
    \[
        e^{c J(t)} \mathcal{N}(u_1(t),u_2(t),u_3(t))
        =
        \Big(w \ast (e^{c J(t)}u_1(t) \overline{e^{-c J(t)}u_2(t)})\Big)
        e^{c J(t)} u_3(t).
    \]
    We recall that $w (x) = \langle x \rangle^{-1} \in L^{\infty}(\mathbb{R})$. Hence,
    \begin{align*} 
        \| e^{c J(t)} \mathcal{N}(u_1(t), u_2(t) , u_3(t))\|_{L^2} &\le 
        \| w \|_{L^{\infty}} 
        \| e^{c J(t)} u_1(t) \, \overline{e^{-c J(t)} u_2(t)} \|_{L^1} 
        \| e^{c J(t)} u_3(t) \|_{L^2}\\
        &\le 
        \| w \|_{L^{\infty}} 
        \| e^{c J(t)} u_1(t) \|_{L^2}
        \| e^{-c J(t)} u_2(t) \|_{L^2}
        \| e^{c J(t)} u_3(t) \|_{L^2}.
    \end{align*}
    Taking $c = \pm \lambda_0$, we conclude that
    \begin{equation}\label{est:LWP-nonlinear}
    \begin{aligned}
        \| e^{\lambda_0 |J(t)|} \mathcal{N}(u_1(t), u_2(t), u_3(t))\|_{L^2}
        &\lesssim
        \| e^{\lambda_0 J(t)} \mathcal{N}(u_1(t), u_2(t), u_3(t)) \|_{L^2} + \| e^{-\lambda_0 J(t)} \mathcal{N}(u_1(t), u_2(t), u_3(t)) \|_{L^2} \\
        &\lesssim
        \| w \|_{L^{\infty}} 
        \prod_{j=1}^3\| e^{\lambda_0 |J(t)|} u_j(t) \|_{L^2}
    \end{aligned}
    \end{equation}
    for some universal constant $C>0$.
    For given $T_0>0$ and $\lambda_0 >0$, define a Banach space $X_{T_0}$ by
    \[
        X_{T_0} 
        = 
        \left\{ 
        u \in C([0,T_0];L^2): \| u \|_{X_{T_0}}:= \sup_{t \in[0,T_0]} \| e^{\lambda_0 |J(t)|} u(t) \|_{L^2} <\infty
        \right\}.
    \]
    Define $\Gamma$ by
    \[
        \Gamma[u](t) = e^{it\Delta/2} u_0 - i \mu\int_0^t e^{i(t-\tau)\Delta/2} N(u(\tau)) \, d\tau,
    \]
    and set $R_0 = \| e^{\lambda_0 |x|}u_0\|_{L^2}$. We prove that $\Gamma$ is a contraction on the closed ball 
    \[
    B_{2R_0}= \{ u \in X_{T_0} : \| u \|_{X_{T_0}} \le 2R_0\},
    \]
    provided that $R_0$ is sufficiently small (depending on $T_0$).
    Indeed, setting $u_1=u_2=u_3=u$ in \eqref{est:LWP-nonlinear}, we get
    \begin{equation*}
    \begin{aligned}
        \| \Gamma[u]\|_{X_{T_0}} 
        &\le
        \| e^{\lambda |J(t)|} e^{it\Delta/2} u_0 \|_{L^2}
        + |\mu|\int_0^t \| e^{\lambda |J(t)|} e^{i(t-\tau)\Delta/2} N(u(\tau)) \|_{L^2} \, d\tau \\
        &\le
        R_0 + |\mu|C_0 T_0 \| w \|_{L^{\infty}}\| u \|_{X_{T_0}}^3 \\
        &\le R_0 + 8|\mu|C_0T_0\|w\|_{L^{\infty}} R_0^3
    \end{aligned}
    \end{equation*}
    for some universal constant $C_0>0$, noting that $J(t) e^{i(t-\tau)\Delta/2} =e^{i(t-\tau)\Delta/2} J(\tau)$. Therefore, $\Gamma$ maps $B_{2R_0}$ to itself, provided that $8|\mu|CT_0 \|w\|_{L^{\infty}} R_0^2 \le 1$.
    Moreover, for $u^{(1)}$, $u^{(2)} \in B_{2R_0}$, we observe that
    \begin{equation*}
    \begin{aligned}
        &\Gamma[u^{(1)}(t)] - \Gamma[u^{(2)} (t)] \\
        &= 
        -i \mu \int_0^t  e^{i(t-\tau)\Delta/2} \left( N(u^{(1)}(\tau)) - N(u^{(2)}(\tau)) \right) \, d \tau \\
        &=
        -i \mu \int_0^t e^{i(t-\tau)\Delta/2} \left( 
        \mathcal{N}(u^{(1)}-u^{(2)}, u^{(1)}, u^{(1)}) 
        + \mathcal{N}(u^{(2)}, u^{(1)}-u^{(2)}, u^{(1)})
        + \mathcal{N}(u^{(2)}, u^{(2)}, u^{(1)}- u^{(2)})
        \right)  \, d\tau.
    \end{aligned}
    \end{equation*}
    Applying \eqref{est:LWP-nonlinear} again, we get
    \begin{align*}
        \| \Gamma[u^{(1)}] - \Gamma[u^{(2)}]\|_{X_{T_0}}
        &\le 
        |\mu|C_1 T_0 \| w \|_{L^{\infty}}
        \left( 
        \| u^{(1)} \|_{X_{T_0}}^2 
        + \| u^{(2)} \|_{X_{T_0}}^2 \right)
        \| u^{(1)} - u^{(2)} \|_{X_{T_0}} \\
        &\le
        8|\mu|C_1 T_0 \|w \|_{L^{\infty}} R_0^2 \| u^{(1)} - u^{(2)} \|_{X_{T_0}},
    \end{align*}
    for some universal constant $C_1 >0$. We conclude that $\Gamma$ is a contraction on $B_{2R_0}$, provided that $|\mu|T_0 R_0^2$ is sufficiently small. Finally, since $s \ge 0$ and $\lambda_0 > \lambda_1 >0$ are given, we have $e^{\lambda_1 r }  \langle r \rangle^s\le C e^{\lambda_0 r}$ for $r\ge 0$, where $C>0$ is a constant depending on $s$ and $\lambda_0-\lambda_1$. This gives
    \[
        \| e^{\lambda_1 |J(t)|} \langle J(t) \rangle^s  u(t) \|_{L^2} \le C \| e^{\lambda_0|J(t)|} u(t) \|_{L^2}.
    \]
    This completes the proof of Lemma \ref{lem:LWP}.
\end{proof}

We now apply Lemma \ref{lem:LWP} with $\lambda_0= \sigma$ and $\lambda_1 = \sigma /2$, where $\sigma >0$ is given in Theorem \ref{thm:main}.
Since $\Phi(1)=0$ and $h(1)=0$, we have $a(1)=b(1)$ and hence
\[
    \| a(1) \|_{\lambda_1, s} 
    =
    \|b(1) \|_{\lambda_1,s} 
    =
    \sqrt{2\pi} \| e^{\lambda_1 |J(1)|}\langle J(1) \rangle^s u(1) \|_{L^2_x}.
\]
By Lemma \ref{lem:LWP}, we get
\[
    \|a(1)\|_{\lambda_1, s} + \| h(1) \|_{\lambda_1, s} = \| a(1) \|_{\lambda_1, s} \le C_* \varepsilon
\]
for some $C_*>0$. We set $\epsilon = C_* \varepsilon$.
From now on, we restrict our analysis to $t \ge 1$.

\section{Nonlinear estimates}\label{sec:nonlinear-estimates}

Fix $T>1$, and let $(a,h)$ be a sufficiently regular solution of $\eqref{eqn:a-2}$ and $\eqref{eqn:phi-2}$ on $[1,T]$.
We define the generator function
\[
    G(t,z) := \| a (t) \|_{z,s}^2 +\|h(t) \|_{z,s}^2.
\]
We remark that this can be represented as
\[
    G(t,z)
    =
    \sum_{m=0}^{\infty} \frac{(2z)^m}{m!}
        \left(
        \| |\partial_v|^{m/2} a(t)\|_{H^s_v}^2 + \| |\partial_v|^{m/2} h(t) \|_{H^s_v}^2 
        \right).
\]
We also note that
\[
    \partial_z G(t,z) 
    = 
    2 \| |\partial_v |^{1/2} a(t) \|_{z,s}^2
    +
    2 \| |\partial_v |^{1/2} h(t) \|_{z,s}^2.
\]

\subsection{Energy estimates}
In this section, we perform energy estimates in analytic norms.

\begin{proposition}\label{prop:energy-a}
	Let $z\ge 0$ and $s>1$. Then we have
	\begin{equation}
		\frac{d}{dt} \| a(t) \|_{z,s}^2
		\lesssim \frac{\Theta(t)^{1+\delta}}{t^2} G(t,z)^{1/2} \left(G(t,z) + \partial_z G(t,z) \right) + \frac{\Theta(t)^{2+2\delta}}{t^2} G(t,z)^2
	\end{equation}
    for $1 \le t \le T$.
\end{proposition}
\begin{proof}
	We compute that
	\begin{equation}\label{identity:energy-a-1}
		\frac{1}{2} \frac{d}{dt} \| a(t) \|_{z,s}^2 
		= \Im \langle i\partial_t a(t) , a(t) \rangle_{z,s}.
	\end{equation}
	Using the equation \eqref{eqn:a-2} for $a(t)$, we get
	\begin{equation}\label{identity:energy-a-2}
		\Im \langle i\partial_t a(t), a(t) \rangle_{z,s}
		= 
		\frac{\Theta(t)^{1+\delta}}{t^2}\Re \langle h\partial_v a + \frac{1}{2} (\partial_v h ) a , a \rangle_{z,s}
		+ \frac{\Theta(t)^{2+2\delta}}{2t^2} \Im \langle h^2 a , a \rangle_{z,s},
	\end{equation}
	noting that $\Im \langle \partial_v^2 a, a \rangle_{z,s}=0$. Using Lemma \ref{lem:analyticspace}, we get
	\begin{align*}
		\left| \Re \langle h\partial_v a + \frac{1}{2} (\partial_v h)a , a \rangle_{z,s}\right|
		\lesssim 
		\left\| h\partial_v a + \frac{1}{2} (\partial_v h) a \right\|_{z,s-1/2}  \| a \|_{z,s+1/2}.
	\end{align*}
	Since $s>1$, the analytic function space $\mathcal{G}^{z,s-1/2}(\mathbb{R})$ is an algebra. Hence,
	\[
		\left\| h\partial_v a + \frac{1}{2} (\partial_v h) a \right\|_{z,s-1/2}
		\lesssim \| h \|_{z,s-1/2} \| a \|_{z,s+1/2} + \| h \|_{z,s+1/2} \| a \|_{z,s-1/2}.
	\]
	We thus obtain
	\begin{align*}
	\left| \Re \langle h\partial_v a + \frac{1}{2} (\partial_v h)a , a \rangle_{z,s}\right|
	&\lesssim 
    \| h \|_{z,s-1/2} \| a \|_{z,s+1/2}^2 
	+ \| h \|_{z,s+1/2} \| a \|_{z,s-1/2} \| a \|_{z,s+1/2}\\
	&\lesssim
	\left( \| h \|_{z,s-1/2} + \| a \|_{z,s-1/2} \right)
	\left( \| h \|_{z,s+1/2}^2 + \| a \|_{z,s+1/2}^2 \right).
	\end{align*}
	From the definition of the generator function, we have
	\[
	\| a\|_{z,s-1/2} + \| h \|_{z,s-1/2} \lesssim G(t,z)^{1/2}.
	\]
    Using
    \[
        \| f\|_{z,s+1/2}^2 \lesssim \| f \|_{z, s}^2 + \| |\partial_v|^{1/2} f \|_{z,s}^2,
    \]
    we obtain that
    \[
    	\| h \|_{z,s+1/2}^2 + \| a \|_{z,s+1/2}^2 \lesssim G(t,z) + \partial_z G(t,z).
    \]
    Combining these bounds, we get
    \begin{equation}\label{est:energy-a-1}
		\left| \Re \langle h\partial_v a + \frac{1}{2} (\partial_v h)a , a \rangle_{z,s}\right|
		\lesssim 
		G(t,z)^{1/2} \left( G(t,z) + \partial_z G(t,z) \right)
	\end{equation}
	Similarly, we bound
	\begin{equation}\label{est:energy-a-2}
		\left| \Im \langle h^2 a, a \rangle_{z,s} \right|
		\lesssim 
		\| h^2 a \|_{z,s} \| a \|_{z,s} 
		\lesssim 
		\| h\|_{z,s}^2 \| a \|_{z,s}^2
		\lesssim G(t,z)^2.
	\end{equation}
	Combining \eqref{identity:energy-a-2}, \eqref{est:energy-a-1}, and \eqref{est:energy-a-2}, we obtain from \eqref{identity:energy-a-1} that
	\[
	\frac{d}{dt} \| a(t) \|_{z,s}^2
	\lesssim \frac{\Theta(t)^{1+\delta}}{t^2} G(t,z)^{1/2} \left(G(t,z) + \partial_z G(t,z) \right) + \frac{\Theta(t)^{2+2\delta}}{t^2} G(t,z)^2.
	\]
	This completes the proof of Proposition \ref{prop:energy-a}.
\end{proof}

Next, we perform energy estimates for $h(t)$.
\begin{proposition}\label{prop:energy-h}
    Let $z\ge 0$ and $s>1$. Then we have
	\begin{equation}
		\frac{d}{dt} \| h(t) \|_{z,s}^2 + 2(1+\delta)\frac{\Theta'(t)}{\Theta(t)} \| h(t) \|_{z,s}^2 \lesssim \frac{1}{t\log^{1+2\delta}(1+t)} G(t,z)^{1/2} \left(G(t,z) +\partial_z G(t,z) \right)
	\end{equation}
    for $1\le t\le T$.
\end{proposition}
\begin{proof}
	We compute
	\begin{equation}\label{identity:energy-h-1}
		\frac{1}{2} \frac{d}{dt}
		\| h(t) \|_{z,s}^2
		=
		\Re \langle \partial_t h(t) , h(t) \rangle_{z,s}.
	\end{equation}
	Using the equation \eqref{eqn:phi-2} for $h$, we get
	\begin{equation}\label{identity:energy-h-2}
	\begin{aligned}
		\Re \langle \partial_t h(t) , h(t) \rangle_{z,s}
		= -(1+\delta)\frac{\Theta'(t)}{\Theta(t)}\| h(t)\|_{z,s}^2
		+  \frac{\mu}{2\pi} \Theta(t)^{-1-\delta} \Re \langle   \partial_v (W_t \ast_v |a(t)|^2) , h(t) \rangle_{z, s}.
	\end{aligned}
	\end{equation}
	It remains to control $\Re \langle \partial_v (W_t \ast_v |a(t)|^2),  h(t) \rangle_{z,s}$. Note that
	\[
		\widehat{W}_t (\eta) = \frac{2}{t} K_0 \left( \frac{|\eta|}{t} \right),
	\]
	and $K_0$ is the modified Bessel function of the second kind. It is known that
	\[
		K_0(r)
		\lesssim
		\begin{cases}
			1+|\log r|, 	&0 < r \le 1, \\
			e^{-r} r^{-1/2}, \quad	& r \ge 1.
		\end{cases}
	\]
	These bounds imply that
	\[
		|\eta| |\widehat{W}_t (\eta)| 
		\lesssim 
		t^{-1} \log(1+t) \langle \eta \rangle, \qquad t \ge 1.
	\]
	Using this bound and Lemma \ref{lem:analyticspace}, we obtain that
	\begin{equation*}
	\begin{aligned}
		\left|\langle \partial_v (W_t \ast_v |a(t)|^2), h(t) \rangle_{z,s} \right|
		&\lesssim
		\left| \int e^{2z|\eta|} \langle \eta \rangle^{2s} \eta \widehat{W}_t(\eta) \widehat{|a|^2}(\eta) \overline{\widehat{h}(\eta)} \, d\eta  \right| \\
		&\lesssim
		\frac{\log (1+t)}{t}  \int e^{2z|\eta|} \langle \eta \rangle^{2s+1}  \left|\widehat{|a|^2}(\eta) \right| \left|\overline{\widehat{h}}(\eta) \right| \, d\eta \\
		&\lesssim 
		\frac{\log (1+t)}{t} \left\| |a(t)|^2 \right\|_{z,s+1/2} \| h(t) \|_{z,s+1/2}.
	\end{aligned}
	\end{equation*}
	Now we may bound
	\begin{align*}
		\left\| |a(t)|^2 \right\|_{z,s+1/2} \| h(t) \|_{z,s+1/2}
		&\lesssim 
		\| a(t) \|_{z,s+1/2} \| a(t) \|_{z,s} \| h(t) \|_{z,s+1/2} \\
		&\lesssim 
		\| a(t) \|_{z,s} \left( \|a(t) \|_{z,s+1/2}^2 + \| h (t) \|_{z,s+1/2}^2 \right) \\
		&\lesssim
		G(t,z)^{1/2} \left(G(t,z) + \partial_z G(t,z) \right).
	\end{align*}
	We thus obtain that
	\begin{equation}\label{est:energy-h-1}
		\left| \langle \partial_v (W_t \ast |a(t)|^2) , h(t) \rangle_{z,s} \right| 
		\lesssim \frac{\log (1+t)}{t} G(t,z)^{1/2} \left(G(t,z) + \partial_z G(t,z) \right).
	\end{equation}
	Combining \eqref{identity:energy-h-1}, \eqref{identity:energy-h-2}, and \eqref{est:energy-h-1}, we get
	\[
		\frac{d}{dt} \| h(t) \|_{z,s}^2 + 2(1+\delta)\frac{\Theta'(t)}{\Theta(t)} \| h(t) \|_{z,s}^2 \lesssim \frac{1}{t\log^{1+2\delta}(1+t)} G(t,z)^{1/2} \left(G(t,z) +\partial_z G(t,z) \right)
	\] 
	which completes the proof of Proposition \ref{prop:energy-h}.
\end{proof}

Combining Propositions \ref{prop:energy-a} and \ref{prop:energy-h}, we obtain the following proposition.
We denote
\[
	\beta_1(t) = \frac{1}{t\log^{1+2\delta}(1+t)}, \qquad \beta_2(t) = \frac{\log^{4+4\delta} (1+t)}{t^2}
\]
for notational convenience.
\begin{proposition}\label{prop:energy-combined}
	There exists a constant $C_0= C_0(s,\delta,\mu) >0$ such that
	\begin{equation}\label{est:energy-combined}
		\frac{d}{dt} G(t,z) 
		\le
		C_0\beta_1(t) G(t,z)^{1/2} \left( G(t,z) + \partial_z G(t,z) \right) 
		+ C_0 \beta_2(t)  G(t,z)^2
	\end{equation}
	for $1 \le t \le T$ and $z \ge 0$.
\end{proposition}

\begin{proof}
	From Propositions \ref{prop:energy-a} and \ref{prop:energy-h}, we have
	\begin{align*}
		&\frac{d}{dt} G(t,z) + 2(1+\delta) \frac{\Theta'(t)}{\Theta(t)} \| h(t) \|_{z,s}^2 \\
		&\lesssim 
		\left(\frac{\log^{2+2\delta}(1+t)}{t^2} + \frac{1}{t \log^{1+2\delta}(1+t)} \right) G(t,z)^{1/2} \left(G(t,z) + \partial_z G(t,z) \right)
		+ \frac{\log^{4+4\delta} (1+t)}{t^2 } G(t,z)^2.
	\end{align*}
	Since
	\[
		2(1+\delta) \frac{\Theta'(t)}{\Theta(t)} \| h(t) \|_{z,s}^2 \ge 0, \qquad \textrm{and} \qquad 
		\frac{\log^{2+2\delta}(1+t)}{t^2} \lesssim \frac{1}{t \log^{1+2\delta} (1+t)},
	\]
	we conclude that \eqref{est:energy-combined} holds. This completes the proof of Proposition \ref{prop:energy-combined}.
\end{proof}

\subsection{Analyticity radius and bootstrap proposition}\label{sec:bootstrap}
We now use Proposition \ref{prop:energy-combined} to choose the time-dependent analyticity radius.
Recall that $\delta >0$ and $\lambda_1 >0$ are given. We remark that both $\beta_1$ and $\beta_2$ are integrable over $t\ge 1$. Let $C_0 >0$ be the constant in the Proposition \ref{prop:energy-combined}, and choose $K>4C_0$. Then choose $\epsilon>0$ such that  $K\epsilon \| \beta_1 \|_{L^1(1,\infty)} \le \lambda_1/2$. Then the time-dependent and decreasing analyticity radius
\[
	\lambda(t) := \lambda_1 - K\epsilon \int_1^t \beta_1 (s) \, ds, \qquad t \ge 1
\]
satisfies $\lambda(1) = \lambda_1$ and $\lambda(t) \ge \lambda_1 /2 $ for all $t \ge 1$.
Now we state the bootstrap proposition.

\begin{proposition}\label{prop:bootstrap}
	Fix $\lambda_1 >0$. Let $C_0 >0$ be the constant in the Proposition \ref{prop:energy-combined}. Let $\epsilon>0$ be sufficiently small and suppose that $G(1,\lambda_1)^{1/2} \le \epsilon$.
	Assume that
	\begin{equation}\label{bootstrap-assumption}
		\sup_{t \in [1,T]} G(t,\lambda(t))^{1/2} \le 4 \epsilon.
	\end{equation}
	for some $T \ge 1$. Then the improved estimate
	\begin{equation}\label{bootstrap-improved}
		\sup_{t \in [1,T]} G(t,\lambda(t))^{1/2} \le 2\epsilon
	\end{equation}
	holds.
\end{proposition}

\begin{proof}
Let
\[
	\widetilde{G}(t,z) = G(t,\lambda(t)z), \qquad 0 \le z \le 1.
\]
We note that
\[
	\partial_t \widetilde{G}(t,z) = \partial_t G(t,\lambda(t) z) + \partial_z G(t,\lambda(t)z) \, \lambda'(t) z,
\]
and
\[
	\partial_z \widetilde{G}(t,z) =\partial_z G(t,\lambda(t)z) \, \lambda(t).
\]
For $z \in [0,1]$, we apply Proposition \ref{prop:energy-combined} to get
\[
	\partial_t \widetilde{G}(t,z) - \frac{\lambda'(t) z + C_0 \beta_1(t) \widetilde{G}(t,z)^{1/2}}{\lambda(t)}\partial_z \widetilde{G}(t,z) 
	\le
	C_0\beta_1(t) \widetilde{G}(t,z)^{3/2} 
	+
	C_0 \beta_2(t) \widetilde{G}(t,z)^2.
\]
Since $\widetilde{G}(t,z)$ is increasing in $z$, we obtain that
\begin{equation}\label{est:tilde-G}
	\partial_t \widetilde{G}(t,z) - B(t,z) \partial_z \widetilde{G}(t,z) 
	\le 
	C_0 \beta_1 (t) \widetilde{G}(t,z)^{3/2}
	+
	C_0 \beta_2(t) \widetilde{G}(t,z)^2,
\end{equation}
where
\[
	B(t,z) = \frac{\lambda'(t) z + C_0 \beta_1(t) \widetilde{G}(t,1)^{1/2}}{\lambda(t)}.
\]

From the choice of $\lambda(t)$ and bootstrap assumption, we have
\[
	B(t,0) = \frac{C_0 \beta_1(t) \widetilde{G}(t,1)^{1/2}}{\lambda(t)} \ge 0
\]
and
\[
	B(t,1) = \frac{\lambda'(t) + C_0 \beta_1(t) G(t,\lambda(t))^{1/2}}{\lambda(t)} \le \frac{(-K +4C_0)\epsilon \beta_1(t)}{\lambda(t)} <0
\]
for $K> 4C_0$.

For each $1 \le t \le T$ and $z_* \in (0,1)$, consider a backward characteristic determined by
\[
	z'(\tau) = - B(\tau ,z(\tau)), \qquad z(t) = z_*.
\]
Then $0 < z(\tau) <1$ for all $1 \le \tau \le t$. Since $z$ is continuous, there exists $m <1$ such that $z(\tau) \le m <1$ for all $1 \le \tau \le t$. Note that, for $0< z < 1$, we have
\[
	\partial_z \widetilde{G}(t,z) 
	= \frac{\lambda(t)}{\pi} \int |\eta| e^{2\lambda(t) z|\eta|} \langle \eta \rangle^{2s} \left( |\widehat{a}(t, \eta)|^2 + |\widehat{h}(t, \eta)|^2 \right) \, d\eta.
\]
Using
\[
	\lambda(t) |\eta| e^{-2\lambda(t)(1-z)|\eta|} \lesssim \frac{1}{1-z},
\]
we obtain that
\[
	\partial_z \widetilde{G}(t, z) \lesssim \frac{1}{1-z} \widetilde{G}(t, 1) <\infty.
\]
Hence, \eqref{est:tilde-G} is justified for $z=z(\tau)$, which yields
\[
	\frac{d}{d\tau} \widetilde{G}(\tau,z(\tau)) 
	\le 
	C_0 \beta_1(\tau) \widetilde{G}(\tau,z(\tau))^{3/2}
	+
	C_0 \beta_2(\tau) \widetilde{G}(\tau,z(\tau))^2.
\]
Observe that
\[
	\widetilde{G}(1,z(1)) = G(1,\lambda(1)z(1)) \le G(1,\lambda_1) \le \epsilon^2
\]
from the assumption on the initial data. Also,
\[
	\widetilde{G}(\tau, z(\tau)) \le 16 \epsilon^2
\]
for $1 \le \tau \le t$ under the bootstrap assumption. Since $\beta_1$ and $\beta_2$ are integrable over $[1,\infty)$, there exist some positive $C_1$ and $C_2$ satisfying
\begin{align*}
	\widetilde{G}(t,z_*) 
	&\le  
	\widetilde{G}(1,z(1)) 
	+ \int_1^t \left( C_0\beta_1(\tau) \widetilde{G}(\tau,z(\tau))^{3/2} + C_0 \beta_2(\tau) \widetilde{G}(\tau,z(\tau))^2 \right) \, d\tau \\
    &\le \epsilon^2 + C_1 \epsilon^3 + C_2 \epsilon^4
\end{align*}
for $1 \le t \le T$. Choosing smaller $\epsilon>0$ if necessary, this is bounded by $4\epsilon^2$ uniformly in $z_* \in (0,1)$. Taking $z_* \rightarrow 1$ gives us
\[
	G(t,\lambda(t)) = \widetilde{G}(t,1) \le 4 \epsilon^2
\]
for $1 \le t \le T$. This completes the proof of Proposition \ref{prop:bootstrap}.
\end{proof}

We now close the bootstrap argument. By the local well-posedness theory, see Lemma \ref{lem:LWP}, the bootstrap assumption holds on $[1,T_0]$ for some $T_0>1$. Proposition \ref{prop:bootstrap} improves the bootstrap bound. The standard continuity argument implies that the solution exists globally in time and satisfies
\[
	\sup_{t \ge 1} \left( \| a(t) \|_{\lambda(t),s} + \| h(t) \|_{\lambda(t), s} \right) \lesssim \epsilon.
\]

\subsection{Dispersive decay bounds}
We now derive the pointwise decay estimates for the solution and the associated density.
Since $s>1/2$ and $\lambda(t) \ge 0$, Sobolev embedding yields
\[
	\| a(t) \|_{L^{\infty}_v} \lesssim \| a(t) \|_{H^s_v} \lesssim \| a(t) \|_{\lambda(t),s} \lesssim \epsilon.
\]
Recalling that
\begin{equation}\label{identity:u-decay-a}
	|u(t,x)|^2 = \frac{1}{2\pi t} |a(t,x/t)|^2,
\end{equation}
we obtain that
\[
	\|u(t) \|_{L^{\infty}_x} \lesssim t^{-1/2} \| a(t) \|_{L^{\infty}_v } \lesssim \epsilon t^{-1/2}
\]
uniformly for $t \ge 1$. Interpolating with mass conservation proves \eqref{est:optimaldecay}.

Next, we derive the time decay of spatial derivatives of $|u(t)|^2$. Let $m \ge 0$ be an integer. Writing $|\eta|^m =|\eta|^m e^{-\lambda(t)|\eta|} e^{\lambda(t)|\eta| }$, we bound
\begin{align*}
	\| \partial_v^m a(t) \|_{H^s}^2
	= \frac{1}{2\pi} \int |\eta|^{2m} \langle \eta \rangle^{2s} |\widehat{a}(t,\eta)|^2 \, d\eta
	\lesssim \left( \sup_{r \ge 0} r^{2m} e^{-2r\lambda(t)} \right) \| a (t) \|_{\lambda(t),s}^2 .
\end{align*}
Since $\lambda(t) \ge \lambda_1 /2$, this yields
\[
	\| \partial_v^m a(t) \|_{H^s}  
	\lesssim  m!\left( \frac{C}{\lambda_1} \right)^m \| a(t) \|_{\lambda(t),s}
	\lesssim \epsilon m! \left( \frac{C}{\lambda_1} \right)^m  .
\]
Using Leibniz rule and the fact that $H^s(\mathbb{R})$ is an algebra for $ s > 1/2$, we get
\begin{align*}
	\left\| \partial_v^m |a(t)|^2 \right\|_{L^1_v \cap L^{\infty}_v}
	\lesssim
	\sum_{j=0}^m \binom{m}{j} \| \partial_v^j a(t) \|_{H^s_v} \| \partial_v^{m-j} a(t) \|_{H^s_v} 
	\lesssim
	\epsilon^2 m! \left(\frac{C}{\lambda_1} \right)^m.
\end{align*}

From \eqref{identity:u-decay-a}, we have
\[
	\partial_x^m |u(t,x)|^2 = (2\pi)^{-1}t^{-m-1} \left(\partial_v^m |a(t,v)|^2 \right)\Big\rvert_{v=x/t},
\]
which implies
\[
    \| \partial_x^m |u(t)|^2 \|_{L^p_x} = (2\pi)^{-1} t^{-m-1 + 1/p} \| \partial_v^m |a(t)|^2 \|_{L^p_v}.
\]
Interpolating between $L^1$ and $L^{\infty}$, we conclude that
\[
	\left\| \partial_x^m |u(t)|^2 \right\|_{L^p_x}
	\lesssim \epsilon^2 m! \left(\frac{C}{\lambda_1} \right)^m t^{-m-1+1/p}
\]
for $1\le p\le \infty$, which could be viewed as \textit{phase mixing estimates} \eqref{est:phasemixing}.

\section{Modified scattering}\label{sec:modified-scattering}
In this section, we establish modified scattering and complete the proof of Theorem \ref{thm:main}.
Fix $0 < \lambda' < \lambda_1/2$ and set $\delta':= \lambda_1/2 - \lambda' >0$. Since $\lambda(t) \ge \lambda_1 /2 $ for $t\ge1$, we get $\lambda(t) - \lambda' \ge \delta'$ for $t \ge 1$. Following the argument in the previous section, we observe that
\[
	\| \partial_v^k f \|_{\lambda', s}^2 
	=
	\frac{1}{2\pi} \int |\eta|^{2k} e^{2\lambda'|\eta|} \langle \eta \rangle^{2s} |\widehat{f}(\eta)|^2 \, d\eta
	\lesssim \left( \sup_{r \ge 0} r^k e^{-\delta' r} \right)^2 \int e^{2\lambda(t)|\eta|} \langle \eta \rangle^{2s} |\widehat{f}(\eta)|^2 \, d\eta
	\lesssim
	\| f \|_{\lambda(t),s}^2
\]
for every fixed integer $k \ge 0$ and $f \in \mathcal{G}^{\lambda(t),s}$. Applying these bounds to \eqref{eqn:a-2}, we obtain that
\begin{align*}
	&\| \partial_t a(t)\|_{\lambda', s}\\
	&\lesssim 
	t^{-2} \| \partial_v^2 a (t) \|_{\lambda', s}
	+ t^{-2} \Theta(t)^{1+\delta} \left( \| h(t) \partial_v a(t) \|_{\lambda', s} + \| (\partial_v h(t)) a(t) \|_{\lambda', s} \right)
	+ t^{-2} \Theta(t)^{2+2\delta} \| h(t)^2 a(t) \|_{\lambda', s}\\
	&\lesssim 
	t^{-2} \| a(t) \|_{\lambda(t),s}
	+ t^{-2} \Theta(t)^{1+\delta} \left( \| h(t) \|_{\lambda', s} \| \partial_v a(t) \|_{\lambda', s} + \| \partial_v h(t) \|_{\lambda', s} \| a(t) \|_{\lambda',s} \right)
	+ t^{-2} \Theta(t)^{2+2\delta} \| h(t)\|_{\lambda',s}^2 \| a(t) \|_{\lambda', s} \\
	&\lesssim
	\epsilon t^{-2}
	+ \epsilon^2 t^{-2} \Theta(t)^{1+\delta}
	+ \epsilon^3 t^{-2} \Theta(t)^{2+2\delta}
\end{align*}
for $t \ge 1$. Consequently, for $t_2 \ge t_1 \ge 1$, we get
\[
	\| a(t_2) - a(t_1) \|_{\lambda', s} \le \int_{t_1}^{t_2} \| \partial_{\tau} a(\tau) \|_{\lambda', s}  \, d\tau,
\]
which proves that $a(t)$ is Cauchy in $\mathcal{G}^{\lambda', s}$. Denoting the limit as $t \rightarrow \infty$ by $a_{\infty}$, we obtain for $t \ge 1$ that
\begin{equation}\label{est:at-ainfty-01}
	\| a(t) - a_{\infty} \|_{\lambda', s}
	\lesssim \epsilon t^{-1} \Theta(t)^{2+2\delta}.
\end{equation}
Recalling 
\[
	a(t,v) = e^{i\Phi(t,v) } b(t,v), \qquad u(t,x) = (2\pi it)^{-1/2} e^{ix^2 /(2t)} b(t,x/t),
\]
we conclude that
\[
\left\| e^{i\Phi(t,v)} (2\pi it)^{1/2} e^{-itv^2/2} u(t,tv) - a_{\infty} (v) \right\|_{\lambda', s} \lesssim \epsilon t^{-1}\Theta(t)^{2+2\delta}
\]
uniformly for $t \ge 1$.
Let
\[
    R(t,v) = e^{i\Phi(t,v)} (2\pi it)^{1/2} e^{-itv^2/2} u(t,tv) - a_{\infty}(v),
\]
and
\[
    Q(t,x) = e^{i\Phi(t,x/t)}(2\pi it)^{1/2} e^{-ix^2/(2t)} u(t,x)- a_{\infty}(x/t).
\]
We proved in \eqref{est:at-ainfty-01} that
\[
    \| R(t) \|_{\lambda',s} \lesssim \epsilon t^{-1} \log^{4+4\delta} (1+t).
\]
Since $Q(t,x) = R(t,x/t)$, we obtain that
\[
    \|Q(t) \|_{L^p_x} = t^{1/p} \| R(t) \|_{L^p_v} \lesssim t^{1/p} \|R(t)\|_{\lambda',s} \lesssim \epsilon t^{-1+1/p} \log^{4+4\delta} (1+t)
\]
for $2 \le p \le \infty$.

It remains to identify the large time behavior of the phase $\Phi$.
We recall that
\[
 	\Phi(t,v) 
    = \frac{\mu}{2\pi}\int_1^t (W_\tau \ast_v |b(\tau)|^2)(v) \, d\tau 
    = \frac{\mu}{2\pi}\int_1^t (W_\tau \ast_v |a(\tau)|^2)(v) \, d\tau.
\]
We first establish the following asymptotic decomposition of $W_t \ast q$.
\begin{lemma}\label{lem:decomposition-potential}
	For $t \ge 1$ and $q \in L^1 (\mathbb{R}) \cap W^{1,\infty}(\mathbb{R})$, we have a decomposition
	\[
		W_t \ast q  = \frac{2\log t}{t} q + \frac{1}{t} F[q]+ \frac{1}{t} R_t[q],
	\]
	where
	\[
		F[q](v) := 2\log 2 \, q(v) + \int_{|z|\le1} \frac{q(v-z)-q(v)}{|z|} \, dz  + \int_{|z|>1} \frac{q(v-z)}{|z|} \, dz.
	\]
	Also, there holds
	\[
		\| F[q]\|_{L^{\infty}} \lesssim \| q \|_{L^1} + \| q\|_{W^{1,\infty}}
	\]
	and
	\[
		\| R_t [q] \|_{L^{\infty}} \lesssim t^{-1} \| \partial_v q \|_{L^{\infty}} + t^{-2} ( \|q \|_{L^1} + \| q \|_{L^{\infty}}) 
	\]
	uniformly for $t \ge 1$.
    If $q \in W^{N,1} \cap W^{N+1, \infty}$, the same decomposition holds after applying $\partial_v^m$, $0 \le m \le N$, with $\partial_v^m F[q] = F[\partial_v^m q]$ and $\partial_v^m R_t[q] = R_t [\partial_v^m q]$. The bounds then apply with $q$ replaced by $\partial_v^m q$.
\end{lemma}

\begin{proof}
	We write
	\[
		(W_t \ast q)(v) 
		= 
		\frac{1}{t} \int \frac{q(v-z)}{\sqrt{z^2 + t^{-2}}} \, dz.
	\]
	For $|z| \le 1$, we write 
	\begin{align*}
		&\int_{|z|\le1} \frac{q(v-z)}{\sqrt{z^2 + t^{-2}}} \, dz\\
		&=
		q(v) \int_{|z|\le1} \frac{dz}{\sqrt{z^2 + t^{-2}}} 
		+
		\int_{|z|\le1} \frac{q(v-z)-q(v)}{\sqrt{z^2 + t^{-2}}} \, dz \\
		&=
		q(v) \int_{|z|\le1} \frac{dz}{\sqrt{z^2 + t^{-2}}} 
		+
		\int_{|z|\le1} (q(v-z)-q(v)) \left( \frac{1}{\sqrt{z^2 + t^{-2}}} - \frac{1}{|z|} \right) \, dz
		+
		\int_{|z|\le1} \frac{q(v-z)-q(v)}{|z|} \, dz.
	\end{align*}
	Observe that
	\[
		\int_{|z|\le1} \frac{dz}{\sqrt{z^2 + t^{-2}}} = 2\log t + 2\log 2 + O(t^{-2}),
	\]
	and
	\[
		\left| \int_{|z|\le1} (q(v-z)-q(v)) \left( \frac{1}{\sqrt{z^2 + t^{-2}}} - \frac{1}{|z|} \right)  dz \right|
		\lesssim
		\| \partial_v q \|_{L^{\infty}} \int_{|z|\le1} \left( 1- \frac{|z|}{\sqrt{z^2 + t^{-2}}} \right)  \, dz
		\lesssim t^{-1} \| \partial_v q \|_{L^{\infty}}.
	\]
	For $|z| \ge 1$, we write
	\[
		\int_{|z|>1} \frac{q(v-z)}{\sqrt{z^2 + t^{-2}}} \, dz
		=
		\int_{|z|>1} q(v-z) \left( \frac{1}{\sqrt{z^2 + t^{-2}}} - \frac{1}{|z|} \right) \, dz
		+
		\int_{|z|>1} \frac{q(v-z)}{|z|} \, dz.
	\]
	Noting that
	\[
		\left| \frac{1}{\sqrt{z^2 + t^{-2}}} - \frac{1}{|z|} \right| \lesssim t^{-2} |z|^{-3}
	\]
	for $|z|>1$, we obtain that
	\[
		\left| \int_{|z|>1} q(v-z) \left( \frac{1}{\sqrt{z^2 + t^{-2}}} - \frac{1}{|z|} \right) \, dz \right|
		\lesssim t^{-2} \| q \|_{L^1}.	
	\]
	Combining these estimates, we complete the proof of decomposition and bounds for $q \in L^1 \cap W^{1,\infty}$. In the case when $q \in W^{N,1} \cap W^{N+1, \infty}$, we repeat the argument for $\partial_v^m q$, $0 \le m \le N$, using differentiation under the integrals. On $|z| \le 1$, the difference quotient is controlled by $\| \partial^{m+1}_v q \|_{L^{\infty}}$. On $|z| >1$, we use $\partial_v^m q \in L^1$. The same argument applies to the remainder. This completes the proof of Lemma \ref{lem:decomposition-potential}.
\end{proof}

We recall from \eqref{est:at-ainfty-01} that
\[
    \|a(t) \|_{\lambda',s} + \|a_{\infty} \|_{\lambda',s} \lesssim \epsilon, \qquad 
    \| a(t) - a_{\infty} \|_{\lambda', s} \lesssim \epsilon t^{-1} \log^{4+4\delta} (1+t).
\]
For every integer $m \ge 0$, applying Leibniz rule, H\"older, and Sobolev embedding yields
\[
    \| \partial_v^m(f g) \|_{L^1 \cap W^{1,\infty}} \lesssim \| f \|_{\lambda',s} \| g \|_{\lambda',s}.
\]
Therefore, we get
\begin{equation}\label{est:a2-ainfty2}
    \| \partial_v^m (|a(t)|^2 - |a_{\infty}|^2)\|_{L^1 \cap W^{1,\infty}} 
    \lesssim \left( \|a(t)\|_{\lambda',s} + \| a_{\infty} \|_{\lambda',s} \right)
    \|a(t) - a_{\infty} \|_{\lambda',s}
    \lesssim \epsilon^2 t^{-1} \log^{4+4\delta} (1+t)
\end{equation}
and
\begin{equation}\label{est:a2}
    \| \partial_v^m |a(t)|^2\|_{L^1 \cap W^{1,\infty}} 
    \lesssim
    \|a(t)\|_{\lambda', s}^2
    \lesssim \epsilon^2.
\end{equation}

Applying Lemma \ref{lem:decomposition-potential}, we may write
\[
	\Phi(t,v) 
    =
    \frac{\mu}{2\pi} \int_1^t \left( \frac{2\log \tau}{\tau} |a(\tau,v)|^2 + \frac{1}{\tau} F[|a(\tau)|^2](v) + \frac{1}{\tau} R_\tau[|a(\tau)|^2](v) \right)  d\tau.
\]
We introduce
\[
	\Psi(t,v) 
    =\Phi(t,v) 
    - \frac{\mu}{2\pi}(\log t)^2 |a_{\infty}(v)|^2 
    - \frac{\mu}{2\pi} (\log t) \,   F[|a_{\infty}(v)|^2].
\]
We obtain for $t \ge 1$ that
\begin{align*}
	\partial_t \Psi(t)
	&=
	\frac{\mu}{2\pi} 
    \left(
    W_t \ast |a(t)|^2  - \frac{2\log t}{t} |a_{\infty}|^2 - \frac{1}{t} F[|a_{\infty}|^2]
    \right)\\
	&=
    \frac{\mu}{2\pi} 
    \left(
	\frac{2\log t}{t} \left( |a(t)|^2 - |a_{\infty}|^2 \right)
	+ \frac{1}{t} \left( F[|a(t)|^2-|a_{\infty}|^2] \right)
	+ \frac{1}{t} R_t[|a(t)|^2]
    \right).
\end{align*}
Moreover, for each integer $m\ge0$, we have
\begin{align*}
	\partial_t \partial_v^m \Psi(t)
	&=
	\frac{\mu}{2\pi} 
    \left(
    \frac{2\log t}{t} \partial_v^m \left( |a(t)|^2 - |a_{\infty}|^2 \right)
	+ \frac{1}{t} \left( F[\partial_v^m(|a(t)|^2-|a_{\infty}|^2)] \right)
	+ \frac{1}{t} R_t[\partial_v^m|a(t)|^2]
    \right).
\end{align*}
Using Lemma \ref{lem:decomposition-potential}, we bound
\begin{align*}
	\| \partial_t \partial_v^m \Psi(t) \|_{L^{\infty}}
	&\lesssim
	\frac{\log t}{t} \left\| \partial_v^m \left( |a(t)|^2 - |a_{\infty}|^2 \right) \right\|_{L^{\infty}}
	+ \frac{1}{t} \left(\left\| \partial_v^m ( |a(t)|^2 - |a_{\infty}|^2 )\right\|_{L^1} + \left\| \partial_v^m (|a(t)|^2 - |a_{\infty}|^2)\right\|_{W^{1,\infty}} \right)\\
	&\qquad + \frac{1}{t} \| R_t [\partial_v^m|a(t)|^2] \|_{L^{\infty}}\\
	&\lesssim
	\epsilon^2 t^{-2} \log^{5+4\delta}(1+t).
\end{align*}
For each integer $N \ge 0$, applying the preceding estimates for $0 \le m \le N$ shows that $\Psi(t)$ is Cauchy in $W^{N,\infty}$. Hence, there exists a real-valued function $\Psi_{\infty} \in \bigcap_{N\ge 0}W^{N, \infty}$ such that
\[
	\| \Psi(t) - \Psi_{\infty} \|_{W^{N,\infty}_v} \le C_N \epsilon^2 t^{-1} \log^{5+4\delta} (1+t)
\]
uniformly for $t \ge 1$. Recalling the definition of $\Psi(t)$, the estimate in particular gives us 
\[
	\left\| \Phi(t) - \frac{\mu}{2\pi}(\log t)^2 |a_{\infty}(v)|^2 - \frac{\mu}{2\pi}(\log t) F[|a_{\infty}(v)|^2] - \Psi_{\infty} \right\|_{L^{\infty}_v}
	\lesssim
	\epsilon^2 t^{-1}  \log^{5+4\delta} (1+t)
\]
uniformly for $t \ge 1$.
Moreover, since $\Psi(1) =0$, we have
\begin{equation}\label{est:Psiinfty}
    \| \Psi_{\infty} \|_{W^{N,\infty}} \le C_N \epsilon^2, \qquad N \ge 0.
\end{equation}

Denote $$r(t,v) = \Psi(t,v) - \Psi_{\infty}(v), \qquad \widetilde{\Phi}(t,v) = \frac{\mu}{2\pi}(\log t)^2 |a_{\infty}(v)|^2 +  \frac{\mu}{2\pi}(\log t)F[|a_{\infty}(v)|^2]$$ for notational convenience.
Observe that $r(t)$ is real-valued and we proved
\[
    \|r(t) \|_{L^{\infty}_v} \lesssim \epsilon^2 t^{-1} \log^{5+4\delta}(1+t).
\]
Define $\widetilde{a}_{\infty}(v) = e^{-i\Psi_{\infty}(v)} a_{\infty}(v)$. Since $a(t) = e^{i\Phi(t)} b(t)$, we have
\begin{equation}
\begin{aligned}
    e^{i\widetilde{\Phi}(t)} b(t) -\widetilde{a}_{\infty}
    &=
    e^{-i\Psi_{\infty}} 
    \left( e^{-ir(t)}a(t) - a_{\infty} \right) \\
    &=
    e^{-i\Psi_{\infty}}
    \left( e^{-ir(t)}(a(t)-a_{\infty}) + (e^{-ir(t)}-1) a_{\infty} \right).
\end{aligned}
\end{equation}
Since $s>1/2$, Sobolev embedding and interpolation give
\[
	\|f\|_{L_v^p}
	\lesssim
	\|f\|_{H_v^s}
	\lesssim
	\|f\|_{\lambda',s},
	\qquad 2\le p\le\infty.
\]
Therefore,
\[
	\|a(t)-a_\infty\|_{L_v^p}
	\lesssim
	\epsilon t^{-1}\log^{4+4\delta}(1+t),
	\qquad
	\|a_\infty\|_{L_v^p}\lesssim\epsilon.
\]
Noting that $\Psi_{\infty}$ and $r(t)$ are real-valued, and $|e^{-ir} -1| \le |r|$, we obtain that
\begin{equation}\label{est:modified-scattering}
\begin{aligned}
    \| e^{i\widetilde{\Phi}(t)} b(t) -\widetilde{a}_{\infty} \|_{L^p_v}
    \lesssim
    \| a(t) - a_{\infty} \|_{L^p_v} + \| r(t) \|_{L^{\infty}_v} \| a_{\infty} \|_{L^p_v}
    \lesssim \epsilon t^{-1} \log^{5+4\delta}(1+t).
\end{aligned}
\end{equation}
Since $|\widetilde{a}_{\infty}| = |a_{\infty}|$, we have
\[
    \widetilde{\Phi}(t,v) = \frac{\mu}{2\pi}(\log t)^2 |\widetilde{a}_{\infty}(v)|^2  + \frac{\mu}{2\pi} (\log t ) F[|\widetilde{a}_{\infty}|^2](v).
\]
Since $u(t) = M(t)D(t)b(t)$ and
\[
\|M(t)D(t)f\|_{L_x^p}
=(2\pi)^{-1/2}t^{-1/2+1/p}\|f\|_{L_v^p},
\]
applying \eqref{est:modified-scattering} to
\(f=b(t)-e^{-i\widetilde\Phi(t)}\widetilde a_\infty\)
gives \eqref{eqn:asymptoticformula}.

It remains to prove \eqref{est:modifiedscattering}. Define $u_{\infty} = \mathcal{F}^{-1} \widetilde{a}_{\infty}$. We may write
\begin{equation}\label{eqn:ms}
u(t)-U(t)e^{-i\widetilde\Phi(t,-i\nabla)}u_\infty
=
M(t)D(t)(b(t)-e^{-i\widetilde{\Phi}(t)} \widetilde{a}_{\infty})
+
M(t)D(t)\mathcal{F}
(I-M(t))\mathcal{F}^{-1}e^{-i\widetilde{\Phi}(t)} \widetilde{a}_{\infty}.
\end{equation}
To bound the first term, we use \eqref{est:modified-scattering}. For $2 \le p \le \infty$, this gives
\begin{equation}\label{est:ms-2}
    \|M(t) D(t) (b(t) - e^{-i\widetilde{\Phi}(t)} \widetilde{a}_{\infty}) \|_{L^p_x}
    \lesssim
    t^{-1/2 + 1/p} \|b(t) - e^{-i\widetilde{\Phi}(t)} \widetilde{a}_{\infty} \|_{L^p_v}
    \lesssim
    \epsilon t^{-3/2 + 1/p} \log^{5+4\delta}(1+t).
\end{equation}
It remains to bound the second term. Hausdorff--Young yields
\begin{equation}\label{est:ms-3-1}
\begin{aligned}
    \| M(t)D(t)\mathcal{F} (I-M(t))\mathcal{F}^{-1}e^{-i\widetilde{\Phi}(t)} \widetilde{a}_{\infty}\|_{L^p_x}
    \lesssim t^{-1/2+1/p}\| (e^{ix^2/(2t)}-1) \mathcal{F}^{-1}e^{-i\widetilde{\Phi}(t)} \widetilde{a}_{\infty}\|_{L^{p'}}
\end{aligned}
\end{equation}
From \eqref{est:a2-ainfty2} and \eqref{est:a2}, we have
\[
    \| \partial_v^m |\widetilde{a}_{\infty}|^2 \|_{L^1 \cap W^{1,\infty}}
    =
    \| \partial_v^m |a_{\infty}|^2 \|_{L^1 \cap W^{1,\infty}}
    \lesssim \epsilon^2
\]
for every nonnegative integer $m$. Using this bound on the expression
\[
    \partial_v^m \widetilde{\Phi}(t) 
    =\frac{\mu}{2\pi}(\log t)^2 \partial_v^m |\widetilde{a}_{\infty}(v)|^2 
    + \frac{\mu}{2\pi} (\log t) F[\partial_v^m |\widetilde{a}_{\infty}|^2](v),
\]
we have
\begin{equation}\label{est:tildePhi-derivative}
    \|\partial_v^m \widetilde{\Phi}(t)\|_{L^{\infty}_v}
    \lesssim
    \epsilon^2 \log^2(1+t)
\end{equation}
for each fixed integer $m\ge 0$, with a constant depending on $m$.
On the other hand, using \eqref{est:Psiinfty}, $\| a_{\infty} \|_{\lambda', s} \lesssim \epsilon$, and $\widetilde{a}_{\infty} = e^{-i\Psi_{\infty}} a_{\infty}$, we get
\begin{equation}\label{est:tildeainfty}
    \| \widetilde{a}_\infty \|_{H^N}
    \le C_N \epsilon
\end{equation}
for $N \ge 0$.
Using \eqref{est:tildePhi-derivative}, \eqref{est:tildeainfty}, and Leibniz rule, we get
\[
    \left\|  e^{-i\widetilde{\Phi}(t)} \widetilde{a}_{\infty}\right\|_{H^2_v}
    \lesssim
    \epsilon \log^4 (1+t), \qquad
    \left\|  e^{-i\widetilde{\Phi}(t)} \widetilde{a}_{\infty} \right\|_{H^3_v}
    \lesssim
    \epsilon \log^6 (1+t).
\]
Interpolating these two bounds yields
\[
    \left\|  e^{-i\widetilde{\Phi}(t)} \widetilde{a}_{\infty}\right\|_{H^{5/2 + \eta}_v}
    \lesssim
    \epsilon \log^{5+2\eta} (1+t)
\]
for $0 < \eta \le 1/2$. For $ g= \mathcal{F}^{-1} (e^{-i\widetilde{\Phi}(t)} \widetilde{a}_{\infty})$, Cauchy--Schwarz gives
\[
    \| x^2 g \|_{L^1}
    \le \| \langle x \rangle^{-1/2-\eta} \|_{L^2}
    \| \langle x \rangle^{5/2 + \eta} g \|_{L^2},
    \qquad
    \| x^2 g \|_{L^2}
    \le
    \| \langle x \rangle^{5/2+\eta} g \|_{L^2}.
\]
Now we may bound for $1 \le p' \le 2$ that
\begin{equation}\label{est:ms-3-2}
\begin{aligned}
    \| (e^{-ix^2/(2t)} -1) \mathcal{F}^{-1} e^{-i\widetilde{\Phi}(t)} \widetilde{a}_{\infty} \|_{L^{p'}_x}
    &\lesssim
    t^{-1} \| x^2 \mathcal{F}^{-1} e^{-i \widetilde{\Phi}(t)} \widetilde{a}_{\infty} \|_{L^{p'}_x} \\
    &\lesssim
    t^{-1} \left\|  e^{-i\widetilde{\Phi}(t)} \widetilde{a}_{\infty}  \right\|_{H^{5/2+\eta}_v}
    \lesssim
    \epsilon t^{-1} \log^{5+2\eta} (1+t).
\end{aligned}
\end{equation}
For the fixed $\delta >0$, choose $0 < \eta \le \min \{ 1/2, 2\delta \}$. Combining \eqref{est:ms-2}, \eqref{est:ms-3-1}, and \eqref{est:ms-3-2}, we obtain from \eqref{eqn:ms} that
\[
    \left\| u(t) - e^{it\Delta/2}e^{-i\widetilde{\Phi}(t,-i\partial_x)} u_{\infty} \right\|_{L^p_x}
    \lesssim \epsilon t^{-3/2+1/p} \log^{5+4\delta}(1+t),
\]
for $2 \le p \le \infty$, which proves \eqref{est:modifiedscattering}. This completes the proof of Theorem \ref{thm:main}.

\bibliographystyle{abbrv}

\end{document}